\documentclass[reqno]{amsart}
\usepackage[scale=0.75, centering, headheight=14pt]{geometry}
\usepackage[latin1]{inputenc}
\usepackage[T1]{fontenc}
\usepackage{lmodern}
\usepackage[english]{babel}

\usepackage{stmaryrd}
\usepackage{tikz}
\usepackage{bbm}
\usepackage{amsmath,amssymb,amsfonts,amsthm}
\usepackage{mathtools,accents}
\usepackage{mathrsfs}
\usepackage{xfrac}
\usepackage{array} 
\usepackage{aliascnt}
\usepackage{booktabs} 
\usepackage{array} 

\usepackage{verbatim} 
\usepackage{subfig} 
\usepackage{mathrsfs}
\usepackage{mathrsfs, dsfont}
\usepackage{amssymb}
\usepackage{amsthm}
\usepackage{amsmath,amsfonts,amssymb,esint}
\usepackage{graphics,color}
\usepackage{enumerate}
\usepackage{mathtools,centernot}
\usepackage{microtype}
\usepackage{paralist} 
\usepackage{cases}
\allowdisplaybreaks

\usepackage{braket}
\usepackage{bm}
\usepackage[numbers,sort,compress]{natbib}
\usepackage[citecolor=blue,colorlinks]{hyperref}
\addto\extrasenglish{}

\usepackage{enumerate}
\usepackage{xcolor}

\usepackage{aliascnt}

\makeatletter
\def\newaliasedtheorem#1[#2]#3{
	\newaliascnt{#1@alt}{#2}
	\newtheorem{#1}[#1@alt]{#3}
	\expandafter\newcommand\csname #1@altname\endcsname{#3}
}
\makeatother

\theoremstyle{plain}

\newtheorem{thm}{Theorem}[section]
\newaliasedtheorem{lem}[thm]{Lemma}
\newaliasedtheorem{prop}[thm]{Proposition}
\newaliasedtheorem{claim}[thm]{Claim}
\newaliasedtheorem{cor}[thm]{Corollary}

\theoremstyle{remark}

\theoremstyle{definition}
\newaliasedtheorem{Def}[thm]{Definition}

\newaliasedtheorem{example}[thm]{Example}

\theoremstyle{remark}
\newaliasedtheorem{rem}[thm]{Remark}
\newaliasedtheorem{ex}[thm]{Example}

\newtheorem*{NREM}{Remark}

\numberwithin{equation}{section}

\def\eps{\varepsilon}

\def\R{\mathbb R}
\def\C{{\mathbb C}}
\def\N{{\mathbb N}}
\def\Z{{\mathbb Z}}
\def\T{{\mathbb T}}

\DeclareMathOperator{\re}{Re}

\newcommand{\e}{\varepsilon}

\DeclareMathOperator{\dv}{div}
\DeclareMathOperator{\Lip}{Lip}

\DeclareMathOperator{\curl}{curl}

\DeclareMathOperator{\Ree}{Re}

\DeclareMathOperator{\rank}{rank}
\DeclareMathOperator{\diam}{diam}

\DeclareMathOperator{\Sym}{Sym}

\DeclareMathOperator{\cof}{cof}
\DeclareMathOperator{\loc}{loc}
\DeclareMathOperator{\id}{id}
\DeclareMathOperator{\Reg}{Reg}
\DeclareMathOperator{\Sing}{Sing}

\newcommand{\s}{\mathcal{S}}

\DeclareMathOperator{\spt}{spt}

\DeclareMathOperator{\dist}{d}

\DeclareMathOperator{\ran}{ran}

\title{Hyperbolic regularization effects for degenerate elliptic equations}

\author[X. Lamy]{Xavier Lamy}
\address{Xavier Lamy
	\hfill\break  Institut Universitaire de France (IUF)
	\& Univ Toulouse, CNRS, IMT, Toulouse, France.}
\email{xavier.lamy@math.univ-toulouse.fr}

\author[R. Tione]{Riccardo Tione}
\address{Riccardo Tione
	\hfill\break Università degli Studi di Torino, dipartimento di Matematica "Giuseppe Peano", Via Carlo Alberto 10, 10123 Torino, Italy}
\email{riccardo.tione@unito.it}

\begin{document}

\begin{abstract}
This paper investigates the regularity of Lipschitz solutions $u$ to the general two-dimensional equation $\dv (G(Du))=0$ with highly degenerate ellipticity.\ Just assuming strict monotonicity of the field $G$ and heavily relying on the differential inclusions point of view, 
we establish a pointwise gradient localization theorem and we show that the singular set of nondifferentiability points of $u$ is $\mathcal{H}^1$-negligible.\ As a consequence, we derive new sharp partial $C^1$ regularity results under the assumption that $G$ is degenerate only on curves.\
This is done by exploiting the hyperbolic structure of the equation along these curves, where the loss of regularity is compensated  using tools from the theories of Hamilton-Jacobi equations and scalar conservation laws.\ Our analysis recovers and extends all the previously known results, where the degeneracy set was required to be zero-dimensional.
\end{abstract}

\maketitle

\section{Introduction}

The topic of this work is the regularity of two dimensional Lipschitz solutions to 
the 
elliptic 
equation
\begin{equation}\label{e:G}
\dv(G(Du))=0 \quad\text{ in }B_1\subset\R^2\,,
\end{equation}
in the sense of distributions,
for a general continuous field $G\colon\R^2\to\R^2$
which is strictly monotone:
\begin{align}\label{e:monotG}
(G(b)-G(a),b-a)>0\quad\forall a\neq b\in\R^2\,.
\end{align}
Specifically, we focus on  the fundamental question:
\begin{align}\label{Q}\tag{Q}
\text{are Lipschitz solutions of \eqref{e:G} }C^1,
\text{ or at least partially }C^1\text{ ?}
\end{align}
Here, partially $C^1$ means $C^1$ in the complement of a closed Lebesgue-negligible set.
For a continuous field $G$, \eqref{e:monotG} is, up to a change of sign, the weakest assumption under which Question \eqref{Q} is reasonable.\ Indeed, if $G(b)-G(a)\perp b-a$ for some $b\neq a$,
then there exist Lipschitz solutions $u$ of \eqref{e:G} with wildly discontinuous gradient, namely $u(x)=\left(x,\frac{b+a}{2}\right)+f\left(\left(x,\frac{b-a}{2}\right)\right)$ for any $f\in \Lip(\R)$ with $f'\in \lbrace \pm 1\rbrace$ a.e..

\medskip

Of particular and classical interest are gradient fields
 $G=D f$ with $f$ strictly convex: in that case, 
\begin{equation}\label{e:Df}
\dv(Df(Du))=0 \quad\text{ in }B_1,
\end{equation}
is the Euler-Lagrange equation
satisfied by minima of the functional
\begin{equation}\label{e:intf}
\int_{B_1}f(Du)\, dx.
\end{equation}
It is well-known that solutions are smooth provided $f$ 
is smooth and uniformly elliptic, that is, 
\begin{equation}\label{e:qua}
\Lambda^{-1}\id \le D^2f(\xi) \le \Lambda \id, \text{ in the sense of quadratic forms}.
\end{equation}
In higher dimensions, this is due to the celebrated theorems of De Giorgi-Nash-Moser \cite{DeGiorgi1957,Nash1958,Moser1960}, while in two dimensions it has been known since the work of Morrey \cite{morrey38}.
Many important and natural problems do not, however, enjoy the strong ellipticity
 property \eqref{e:qua}.
A prominent example is that of the $p$-Laplace equation corresponding to $f(\xi)=|\xi|^p$.
More generally, 
problems where \eqref{e:qua} fails 
at a polynomial rate
for $|\xi|\to 0$ and $|\xi|\to \infty$
have been studied extensively, see the review \cite{MingioneDarkSide}.
We are aware of very few results going beyond this type of degeneracy in higher dimensions: see \cite{Mooney2020,Marino2023,lledos24},
and \cite{CF14a,CF14b} for related regularity questions when $f$ is convex but not strictly convex.

Note that, due to the Lipschitz assumption, \eqref{Q} focuses
on controlling \emph{oscillations} of the gradient $Du$.
A similar question could be asked about \emph{concentration} effects, namely starting from a $W^{1,p}(B_1)$ function.\
Some results in this direction, valid in all dimensions, can be found in \cite{DeFilippis2023,DeFilippis2025} and references therein.\
Recent counterexamples  \cite{CT,Johansson2024} emphasize the necessity of careful assumptions on $u$ and $G$,
but the Lipschitz property of $u$ can be inferred naturally in many situations, 
see e.g. \cite[\textsection 12.4]{Gilbarg1977} and \cite{Chipot1986}.

\subsection{Previous two-dimensional results}

Morrey's original result \cite{morrey38} admits a first far-reaching generalization to $\delta$-monotone fields, which satisfy a more quantitative version of \eqref{e:monotG}, see \cite[\textsection 16]{Astala2008}.\ This covers for instance the $p$-Laplace equation, but also many more degenerate cases: for instance \eqref{e:qua} may fail on a large set of $\xi$'s as long as the eigenvalues of $D^2f(\xi)$ are comparable.\

\medskip

More generally and, to the best of our knowledge, for the first time, regularity results in the absence of any quantitative assumptions on the strict monotonicity \eqref{e:monotG} have been established by D.~De Silva and O.~Savin
in  \cite{dSS10}.
For any strictly convex $f$, 
their results depend only on the closed sets $\mathcal D_-$ and $\mathcal D_+$ 
of  values $\xi\in\R^2$ 
where ellipticity fails from below or from above:
\begin{align*}
\mathcal D_-
&
=\bigcap\Big\lbrace X \subset\R^2\text{ closed}\colon
\exists\lambda >0,\, D^2 f \geq\lambda\id\text{ in }\R^2\setminus X\Big\rbrace\,,
\\
\mathcal D_+
&
=\bigcap\Big\lbrace X \subset\R^2\text{ closed}\colon
\exists\Lambda >0,\, D^2 f \leq\Lambda\id\text{ in }\R^2\setminus X\Big\rbrace\,.
\end{align*}
The inequalities are in the sense of distributions with values into symmetric matrices.
These definitions have natural extensions to any strictly monotone field $G$, see \eqref{e:Dpm}.
It may be useful to have a picture of these  sets in simple examples:
\begin{align*}
(\mathcal D_-,\mathcal D_+)
=
\begin{cases}
(\emptyset,\lbrace 0\rbrace)
&\text{ if }f(\xi)=|\xi|^p\text{ for }1<p<2\,,
\\
(\lbrace 0\rbrace,\emptyset)
&\text{ if }f(\xi)=|\xi|^p\text{ for }p>2\,,
\\
(\lbrace \xi_2=0\rbrace,\lbrace \xi_1=0\rbrace)
&\text{ if }f(\xi)=|\xi_1|^{p_1} +|\xi_2|^{p_2}
\text{ for }1<p_1<2<p_2\,.
\end{cases}
\end{align*}

With this notation, 
the most general answer to \eqref{Q} known so far is, essentially, that
\begin{align}\label{e:Dfinite}
\text{Lipschitz solutions of \eqref{e:G} are }C^1
\text{ if }\mathcal D =\mathcal D_-\cap\mathcal D_+
\text{ is finite.}
\end{align}
Quite remarkably, this only involves the degeneracy set $\mathcal D$ where ellipticity fails 
\emph{both} from above and below.
This regularity result is proved in \cite{Lacombe2024} 
by building on the arguments introduced in \cite{dSS10},
where  $\mathcal D$ was assumed empty, or $\mathcal D_-$ finite.
One can actually replace `finite' in \eqref{e:Dfinite} by `at most countable', with the same proof.
Moreover, an example given in \cite[Theorem~1.5]{Lacombe2024}
demonstrates that, under the mere assumption \eqref{e:monotG}, Lipschitz solutions of \eqref{e:G} may have point singularities,
so one cannot in general expect better than partial regularity in \eqref{Q}.
For different but related problems, some partial regularity results have been obtained in cases where the degeneracy set $\mathcal D$ may be one-dimensional:
the boundary of a convex polygon for solutions of an obstacle problem \cite{dSS10} or a class of smooth curves in which $Du$ is constrained to lie \cite{LLP25}.

\subsection{Main results}

Our first main result is valid for \emph{any} strictly monotone continuous field $G$.\ We denote by $\mathcal{H}^\alpha$ the $\alpha$-dimensional Hausdorff measure.

\begin{thm}\label{t:Su}
Let $G\colon\R^2\to\R^2$ be continuous and strictly monotone \eqref{e:monotG},
and $u\colon B_1\to\R$ a Lipschitz solution to \eqref{e:G}.
Then the nondifferentiability set 
\begin{align*}
\mathcal S_u
=\lbrace x\in B_1\colon u\text{ is not differentiable at }x
\rbrace
\end{align*}
satisfies $\mathcal H^1(\mathcal S_u)=0$,
and $Du$ is continuous on $B_1\setminus \mathcal S_u$.
\end{thm}

Compared to a generic Lipschitz function, 
the gain of information provided by Theorem~\ref{t:Su} is that:
\begin{itemize}
\item the set of non-differentiability points is $\mathcal H^1$-negligible, instead of merely $\mathcal L^2$-negligible;
\item oscillations of $Du$ are controlled pointwise at all differentiability points, 
instead of merely in average at  Lebesgue points of $Du$.
\end{itemize}
Thanks to the latter, partial $C^1$ regularity follows from Theorem~\ref{t:Su} if one establishes that
 $\mathcal S_u$ is closed, or at least that its closure is $\mathcal L^2$-negligible.
Our second main result achieves this, and more,
under structural assumptions on the degeneracy set $\mathcal D=\mathcal D_-\cap\mathcal D_+$
where ellipticity fails both from below and from above.

\begin{thm}\label{t:Dintro}
Let $G\colon\R^2\to\R^2$ be continuous and strictly monotone \eqref{e:monotG}.
Then, for any Lipschitz solution $u$ to \eqref{e:G}, 
the singular set
 $\mathcal S_u\subset B_1$ is
\begin{itemize}
\item  locally finite, if the graph of $G$ over $\mathcal D$
 is included in a finite disjoint union
of $C^1$ curves in $\R^2\times\R^2$;
\item empty, if $\mathcal D$ is included in a countable disjoint union of boundaries of strictly convex sets.
\end{itemize}
\end{thm}

We will actually prove a more general statement, see Theorem~\ref{tint:4}, 
but Theorem~\ref{t:Dintro}
conveys the main novelty of our methods, which is to deal with a large class of one-dimensional degeneracy sets $\mathcal D$, in contrast with the zero-dimensional assumption in \eqref{e:Dfinite}.\ 
The first part of this result should be compared with the example found in \cite[Theorem~1.5]{Lacombe2024}, of $u$ and $G$ solving \eqref{e:G}-\eqref{e:monotG} and $u$ having a point singularity.
In that example, the graph of $G$ over $\mathcal D$ is a smooth curve.
On one hand, it shows therefore
 the sharpness of the first part of Theorem \ref{t:Dintro}.
On the other hand, Theorem~\ref{t:Dintro} shows the optimality of   
\cite[Theorem~1.5]{Lacombe2024} in the sense that point singularities are the only possible irregular phenomenon.
The second part of Theorem~\ref{t:Dintro} applies to a wide class of radial fields  $G=Df$ with $f$ strictly convex and radial, a symmetry present in most physically relevant problems.\ More precisely, 
consider   $f(x)=g(|x|)$ for some increasing and strictly convex $C^1$ function $g\colon [0,\infty)\to [0,\infty)$ 
such that $g(0)=g'(0)=0$. Then we have $\xi\in\mathcal D$ if and only if $|\xi|\in D_g$, where
\begin{align*}
D_g=
\bigcap
\Big\lbrace X\subset [0,\infty)\text{ closed}
\colon
\exists\lambda>0,\;
&
g''\geq \lambda \text{ a.e. in }[0,\infty)\setminus X
\text{ or }
g''\leq \frac 1\lambda 
\text{ a.e. in }[0,\infty)\setminus X
\Big\rbrace\,.
\end{align*}
Informally, one can think of $D_g$ as the set of points $t \ge 0$ 
near which $g''$ can be \emph{both} arbitrarily large and arbitrarily small.
If $D_g$ is totally disconnected, then all connected components of $\mathcal D$ are either circles or the origin, and we deduce that any Lipschitz minimizer of \eqref{e:intf} is $C^1$.\ This is entirely new: in the context of radial functionals, the available techniques so far could only show such regularity under the assumption that $\mathcal{D}$ is precisely the origin.\ We conclude this introduction by explaining the strategy and the tools used to show the above results, other than giving an overview of the structure of the paper.

\subsection{Strategy of Theorem~\ref{t:Su}}

Given some uniform monotonicity in \eqref{e:monotG}, 
a classical first step is to obtain information on $Du$ by choosing suitable test functions and obtaining \emph{integral estimates}.\ 
This step is present for instance in the aforementioned works \cite{morrey38,dSS10,Astala2008,LLP25},
and, more generally, all regularity results we are aware of 
require quantitative forms of the monotonicity \eqref{e:monotG} at least in suitable open sets.
However, to solve Question \eqref{Q} in full generality, 
one needs to deal with \emph{arbitrarily degenerate} monotonicity on \emph{any} open set.
For instance, the degeneracy set $\mathcal D=\mathcal D_-\cap\mathcal D_+$ can fill the whole plane, see Remark~\ref{r:eps*}.

Thus, we need to rely on different and novel tools: the main point in the proof of Theorem~\ref{t:Su} is an $L^\infty$ to $W^{1,\infty}$ estimate inspired by the beautiful separation argument of \cite{FS08}.\ 
It states that pointwise gradient localization into a given convex region is stable with respect to the $L^\infty$ topology,
see  Theorem~\ref{t:noninc}.
Since linear functions solve \eqref{e:G} and their gradient is localized at a single value,
this implies a flatness improvement property:
 if a solution $u$ of \eqref{e:G} with $|Du|\leq L$ is uniformly close to a linear function $\ell_\xi(x)=(\xi,x)$,
then $Du$ is uniformly close to $\xi$ in a smaller ball.
More precisely, $\forall\xi\in\overline{B_L},$
\begin{align}
\label{e:flat}
\forall\eps >0,\,
\exists\delta >0
\colon
\quad
\|u-\ell_\xi\|_{L^\infty(B_1)} \leq\delta
\quad\Rightarrow
\quad
\|D u-\xi\|_{L^\infty(B_\delta)}\leq\epsilon\,.
\end{align}
Analogous flatness results are at the heart of \cite{CF14a,CF14b,Mooney2020}, but with the fundamental difference that they require $\xi$ to be an elliptic value, that is, $\xi\notin \mathcal D_-\cup\mathcal D_+$.\ Similar arguments can be found also in \cite{dSS10,Lacombe2024}, provided that $\xi \notin \mathcal D = \mathcal{D}_-\cap\mathcal D_+$,
but observe that this might be an empty statement in our general context, since $\mathcal D$ could be the whole plane.\
 Here, instead, we can apply \eqref{e:flat} for \emph{any} $\xi\in\R^2$,
and infer in particular that if $Du$ has vanishing mean oscillation (VMO) at $x\in B_1$, 
then $u$ is differentiable at $x$, and $Du$ has vanishing pointwise oscillation at $x$.
This implies Theorem~\ref{t:Su}, provided we show that non VMO points of $Du$ are $\mathcal H^1$-negligible.
This last fact follows from the strict monotonicity of $G$, and is actually valid in any  dimension (with $\mathcal H^1$ replaced by $\mathcal H^{n-1}$),
 see 
 Theorem~\ref{p:sing}.

Note that, from Theorem~\ref{t:Su}, a classical way to show that $\mathcal S_u$ is closed, hence prove partial regularity, would be to combine it with an $\e$-regularity result:
\begin{align}\label{e:epsreg}
\inf_{\xi\in\overline{B_L}}\| Du-\xi\|_{L^\infty(B_1)}\leq\e \quad\Longrightarrow \quad Du\text{ continuous in }B_{\e}\,,
\end{align}
for an $\e>0$ depending only on $G$ and $L$.\ However, there exist continuous strictly monotone fields for which such an $\e>0$ cannot exist, see Remark~\ref{r:eps*}, 
so the full answer to \eqref{Q} would have to follow a different route.

\subsection{Strategy of Theorem~\ref{t:Dintro}}
The starting point is the recent and remarkable localization result of \cite{lacombe26}, which states that any blowup limit
\begin{align*}
u_\infty(x)=\lim_{r_j\to 0} \frac{u(x_0+r_jx)-u(x_0)}{r_j}\,,
\end{align*}
is either linear or satisfies $Du_\infty\in \mathcal D$ a.e. in $\R^2$.
At a singular point, we are in the latter situation, 
and blowup limits are solutions of the constrained problem
\begin{align}\label{e:constrDu}
\dv G(D u_\infty)=0,\text{ and }D u_\infty\in\mathcal D\,.
\end{align}
We use this constrained problem to show that blowup limits are one-homogeneous with a single singularity at the origin: 
standard arguments then imply that $\mathcal S_u$ is locally finite.
It is worth emphasizing that homogeneity of blowup limits is typically obtained in elliptic problems through monotonicity formulas, which are unavailable in our problem.\ Here, in contrast, we obtain this homogeneity relying on:
\begin{enumerate}
\item hyperbolic regularization effects hidden in \eqref{e:constrDu} when $\mathcal D$ is a curve, see  Theorem~\ref{tint:3};
\item the observation of independent interest that general solutions to \eqref{e:G} have a well-defined Monge-Ampère measure $\det(D^2u)\leq 0$,
which concentrates at the origin for blowup limits, see 
Theorem~\ref{thm:u} and Corollary \ref{cor:2}. 
\end{enumerate} 
We will comment further on the second point in \textsection \ref{sec:2bis}, but let us be more specific about the first point, namely the hyperbolic regularization effect,
which yields a strong generalization of the main result of \cite{LLP25}
about the regularity of differential inclusions into curves.
The starting observation is that, when $\mathcal D$ is a curve,
the degenerate elliptic problem \eqref{e:constrDu} 
  becomes a hyperbolic system,
and $C^1$ solutions have constant gradients along characteristic lines.
A typical feature of hyperbolic problems is that
weak solutions can be extremely wild,
but regularization effects are available for a subclass of weak solutions.
Here we manage to exploit such effects and eventually obtain 
an $\e$-regularity result \eqref{e:epsreg} for solutions of \eqref{e:constrDu}.\ An observation from \cite{LLP25} relating \eqref{e:constrDu} to 
zero-energy states of a phase transition model \cite{GMPS24} then enables us to 
deduce that the singular set $\mathcal S_{u_\infty}$ is locally finite.\ For entire solutions of \eqref{e:constrDu}, such as blowup limits, the constancy along characteristics outside $\mathcal S_{u_\infty}$ further enforces a quite rigid structure, which we can finally combine with the concentration property of the Monge-Ampère measure to conclude that blowup limits are one-homogeneous.

\subsection{The differential inclusions approach and nonlinear Beltrami equations}\label{sec:dia} Ubiquitous in our approach is the correspondence between Lipschitz solutions of \eqref{e:G} and solutions to differential inclusions.
\ In order to make a short but comprehensive overview of the subject, we postpone a thorough explanation of this theory to \textsection \ref{sec:ld}.\ Alongside differential inclusions, we will also exploit the equivalent viewpoint of nonlinear Beltrami systems, which will be introduced in \textsection \ref{sec:nbs}.\ There, we will explain in detail the (well-known) interplay and equivalence among these three viewpoints.\ It is important to notice that, on one hand, the theory of $2\times 2$ elliptic differential inclusions is extremely rich, and this wealth of results and techniques will be essential in our paper.\ Conversely, our results, especially Theorem \ref{t:Su}, represent an important step towards a complete understanding of elliptic differential inclusions in $\R^{2\times 2}$, see Question \eqref{QDI} below.

 \subsection*{Structure of the paper} We begin by introducing differential inclusions in \textsection \ref{sec:ld}.\ In \textsection\ref{sec:prel}, we explain the notation and the well-known, but crucial, results mentioned in \textsection\ref{sec:dia}.\ \textsection\ref{sec:sep}-\textsection\ref{sec:size} contain our first two main Theorems \ref{t:noninc}-\ref{p:sing} that, combined, yield Theorem \ref{t:Su}, 
proved in \textsection\ref{sec:proofthm1}.\ In \textsection\ref{sec:2bis} we introduce the Monge-Ampère measure associated to a solution of the elliptic PDE \eqref{e:G}.\ \textsection \ref{sec:curves} contains a sharp result on inclusions into elliptic curves, Theorem \ref{tint:4}.\ All of these results will be used in the proof of Theorem \ref{t:Dintro}, contained in \textsection\ref{s:degen_eq}.\ Some sections are lengthier than others:\ in that case we will describe the section at the start of it.

\subsection*{Acknowledgments}
XL is supported by the ANR project ANR-22-CE40-0006.\ We thank L. De Rosa and A. Guerra for highly valuable comments on the presentation.

 \section{The differential inclusions point of view}\label{sec:ld}
 Before explaining some elements from the theory of differential inclusions and how they relate to our paper, let us set a notation which will be useful throughout: we express condition \eqref{e:monotG} equivalently as
 \begin{align}\label{e:monoG}
 	(G(b)-G(a),b-a)\ge \sigma(|a-b|)\quad\forall a,b\in\R^2\,,
 \end{align}
 where 
 \begin{equation}\label{e:sigmaprop}
 	\sigma: [0,+\infty) \to [0,+\infty) 	\text{ is a continuous, strictly increasing, convex function.}
 \end{equation} 
 Now, given a solution $u \in \Lip(B_1)$ to \eqref{e:G}, one can construct a solution $w = (u,v) \in \Lip(B_1,\R^2)$ to
 \begin{equation}\label{e:di}
 	Dw \in K \text{ a.e. in }B_1,
 \end{equation}
 where $K\subset\R^{2\times 2}$ fulfills, for a possibly different $\sigma$ still fulfilling \eqref{e:sigmaprop},
 \begin{equation}\label{e:diffinc}
 	\sigma(|X-Y|) \le \det(X-Y), \quad \forall X,Y \in K,
 \end{equation}
 as soon as $G$ fulfills the monotonicity conditions \eqref{e:monoG}-\eqref{e:sigmaprop}.\  The converse also holds: given a solution $w = (u,v)$ to \eqref{e:di} with $K$ enjoying properties \eqref{e:diffinc}-\eqref{e:sigmaprop}, then $u$ solves \eqref{e:G} for $G$ satisfying \eqref{e:monoG}-\eqref{e:sigmaprop},
 see Proposition \ref{p:11} below.\ In the following, for brevity we will often say that $G$ (or $K$) fulfills \eqref{e:monoG} (or \eqref{e:diffinc}) to say that $G$ (or $K$) fulfills \eqref{e:monoG} (or \eqref{e:diffinc}) for $\sigma$ fulfilling \eqref{e:sigmaprop}.\ 
 
 \medskip
 For differential inclusions, \emph{ellipticity} can be defined as in \cite[\textsection 5]{sverak93}, where a compact set $K$ is called \emph{elliptic} if it fulfills \eqref{e:diffinc}, it is a manifold and its tangent space at every point has no rank-one connections.\ Recall that $A,B \in \R^{n\times m}$ are rank one-connected if $\rank(A-B) = 1$.\ For consistency with the usual terminology in the world of PDEs, in this paper we refer to this requirement as \emph{uniform ellipticity}. 
 For example, the $p$-Laplacian is an \emph{elliptic} equation which does not admit elliptic linearizations at all points (hence is not uniformly elliptic) if $p\neq 2$.\ Thus, we give a weaker definition of ellipticity than \cite{sverak93}, that a priori includes and extends the one carried by the PDE \eqref{e:G}, where $G$ fulfills \eqref{e:monoG}-\eqref{e:sigmaprop}.\ While this is just a matter of terminology, it is crucial to underline the difference between our meaning of the word \emph{elliptic} and the one of \cite{sverak93}: the fact that our elliptic sets may not admit elliptic tangent spaces at all points is the crux of the matter in Question \eqref{Q}.\ In simple terms, our notion of ellipticity is equivalent to \emph{approximate rigidity} of $K$:
 \begin{Def}
 	A set $K \subset \R^{n\times m}$ is called \emph{elliptic} if, given any equi-Lipschitz sequence $(u_k)_k$ on $B_1$,
 	\begin{equation}\label{eq:app}
 		\dist(Du_k,K) \to 0 \text{ in }\mathcal{D}'(B_1) \quad\Longrightarrow\quad \text{ $(u_k)_k$ is $W^{1,1}_{\loc}$ strongly precompact.} 
 	\end{equation}
 \end{Def}
 \begin{NREM}
 	Equivalently, a compact set $K$ is elliptic if it only supports trivial homogeneous gradient Young measures.\ For a thorough introduction on Young measures, see for example \cite{Mueller1999}.
 \end{NREM}
 The primary examples of elliptic sets are precisely those with property \eqref{e:diffinc}, compare \cite[Theorem 1]{sverak93}:

 \begin{thm}\label{t:comp}
 	If $K \subset \R^{2\times 2}$ is a compact set fulfilling \eqref{e:diffinc}, then $K$ is elliptic.
 \end{thm}

 Showing that a set in $\R^{n\times m}$ is elliptic is, in general, very hard.\ In many situations, for instance in some convex integration schemes \cite{Mueller2003,Szekelyhidi2004}, it is actually sufficient to show that the set is \emph{not} elliptic 
 by finding special subsets of matrices inside $K$: common choices of such special sets are rank-one connections and $T_N$ configurations, see \cite[Definition 1]{Szekelyhidi2007}.\ If $K \subset \R^{2\times 2}$, we have a much clearer picture.\ In particular, from the deep works \cite{Szekelyhidi2007,FS08} we have the following very simple way of deciding whether a given set in $\R^{2\times 2}$ is elliptic.
 
 \begin{thm}\label{thm:charac}
 	Let $K \subset \R^{2\times 2}$ be a compact set. Then $K$ is elliptic if and only if $K$ does not contain rank-one connections and $T_4$ configurations.
 \end{thm}
 
 This result sheds a much clearer light on the notion of ellipticity for \eqref{e:di}, since it provides an algorithmic way of checking it.\ We can then reformulate Question \eqref{Q} in the language of differential inclusions:
 \begin{equation}\label{QDI}\tag{Q:DI}
 	\text{let $K$ be compact and elliptic. Are solutions to \eqref{e:di} $C^1$, or at least partially regular?}
 \end{equation}
 
 A deep fact is that Question \eqref{QDI} can be completely reduced to its, a priori, simpler counterpart Question \eqref{Q}.\ This is done in two steps:
 
 \begin{thm}[\cite{KS08}]\label{theo:connected}
 	Let $K \subset \R^{2\times 2}$ be a compact elliptic set.\ Then, the gradient of any solution $w \in \Lip(B_1,\R^2)$ to \eqref{e:di} has connected essential range.
 \end{thm}
 
 Here, if $\Omega \subset \R^m$ is an open set, for a map $f\in L^1(\Omega,\R^N)$ we denote by $[f](\Omega')$ its \emph{essential range in $\Omega'\subset\Omega$}, namely the smallest closed set with the property that 
 \begin{equation}\label{e:egd}
 	\text{$f(x) \in [f](\Omega')$ for a.e. $x \in \Omega'$.} 
 \end{equation}
 We will write $[f]$ for $[f](\Omega)$ (and call it simply the \emph{essential range} of $f$) and we also set:
 \begin{equation}\label{e:point}
 	[f](x_0) \doteq [f]\left(\bigcap_{r>0}B_r(x_0)\right).
 \end{equation}
We will add a few words on the essential range of the gradient of a map in \textsection \ref{sec:ergd}.\
 
 \begin{rem}
 	The statement of \cite[Theorem 1]{KS08} is more general than Theorem~\ref{theo:connected}.
 	It asserts  that, for any $w \in \Lip(B_1,\R^2)$, the \emph{rank-one convex hull}  $[Dw]^{\text{rc}}$
 	is connected.
 	This implies Theorem~\ref{theo:connected}
 	because  $[Dw]^{\text{rc}}=[Dw]$ if $[Dw]\subset K$ for an elliptic set $K$.\ Indeed, the ellipticity assumption ensures that the quasiconvex hull of $[Dw]$, $[Dw]^{\text{qc}}$, which fulfills $[Dw] \subset [Dw]^{\text{rc}} \subset [Dw]^{\text{qc}}$, is equal to $[Dw]$, 
 	see \cite[\textsection{4.4}]{Mueller1999}.
 \end{rem}

 To show how to reduce Question \eqref{QDI} to Question \eqref{Q} we need a second, and last, ingredient:
 
 \begin{lem}[\cite{sverak93}]
 	Let $K \subset \R^{2\times 2}$ be a compact, connected set without rank-one connections. Then there exists $ c \in \R$ such that $c\det(X-Y) > 0$, $\forall X,Y \in K, X\neq Y$.
 \end{lem}
 
 In other words, 
 up to a 
change of orientation,
 a compact, connected, elliptic set fulfills \eqref{e:diffinc}.\ 
 This, combined with Theorem \ref{theo:connected}, yields that, if $w$ is a Lipschitz solution to \eqref{e:di} for an elliptic set $K$, then 
 (up to passing to a connected component)
  we can assume that $K$ fulfills \eqref{e:diffinc}.\
Thus there is no real difference between Questions \eqref{Q} and \eqref{QDI}.\ This represents a weak converse to Theorem~\ref{t:comp}.\
 
 
\section{Notation and preliminaries}\label{sec:prel}

Let us first recall the notation and some results we will use in the course of our paper.

\subsection{General notation and some elementary facts}

\subsubsection*{Topology}

Let $E \subset \R^n$ be any set. Then, $\overline{E}$ denotes its closure, $\partial E$ its topological boundary, $E^c$ its complement in $\R^n$ and $\diam(E)$ its diameter. For two sets $A,B$, we denote by $\dist(A,B)$ the distance between them.\ Moreover, $A \Subset B$ means that $\overline{A} \subset B$.\ The open ball of radius $r$ centered at $X$ in $\R^n$ or in $\R^{n\times m}$ is denoted by $B_r(X)$.\ If $X = 0$, we will simply write $B_r$.\ Finally, given a Lipschitz function $f: E \subset \R^m \to \R^n$, we write $\Lip(f)$ for its Lipschitz constant.

\subsubsection*{Measure theory}
$|E|$ denotes the Lebesgue measure of a measurable set $E \subset \R^n$.
We write $\mathbf 1_E$ for the indicator function of $E$, namely $\mathbf 1_E(x) = 1$ if $x \in E$ and $\mathbf 1_E(x) = 0$ otherwise.\ In the case of super-level sets of a function $u: F \subset \R^n \to \R$ (and analogously for sub-level sets and for preimages of intervals) we will use expressions such as $\mathbf{1}_{\{u > \alpha\}}$ to denote the function that, at $x \in F$, returns $1$ if $u(x)>\alpha$ and $0$ otherwise.\
If $|E| < +\infty$, $\fint_E f(x) dx$ denotes the average of $f$ over $E$.\ In most cases we will consider only $L^p$ spaces with respect to the Lebesgue measure.\
Otherwise, we will write $L^p(\Omega,X;\mu)$ for the space of $L^p$ functions with respect to the measure $\mu$ on $\Omega$ and with values in the set $X$.


\subsubsection*{Linear Algebra}
We use the notation $e_i$ for the vectors of the canonical basis of $\R^2$, $e_1 = (1,0)$ and $e_2 = (0,1)$. For a matrix $A \in \R^{n\times m}$, we denote by $L_A$ the linear map associated to $A$, $L_A(x) \doteq Ax$.\ The image of $L_A$ is the \emph{range} of $A$, $\ran A$. $\det(A)$, $A^T$ and $|A|$ denote the determinant, the transpose and the Euclidean norm of the matrix $A$, respectively.\ The (standard) scalar product between matrices is denoted by $\langle A,B\rangle$, while for vectors $a,b$ we use $(a,b)$ or, if confusion may arise, $a\cdot b$.\ The cofactor matrix is
\begin{equation}\label{eq:A}
	\cof(A) \doteq \left(\begin{array}{cc} d & -c \\ -b & a\end{array}\right), \quad \text{ if } A = \left(\begin{array}{cc} a & b \\ c & d\end{array}\right),\quad
	\text{so that }\cof^T(A)A = A\cof^T(A)= \det(A)\id\,.
\end{equation}

\subsubsection*{Complex derivatives}\label{sec:code}

It will be convenient to use \emph{complex} notation to represent matrices $A \in \R^{2\times 2}$. 
Let $A$ be as in \eqref{eq:A}.\ Its conformal and anti-conformal parts are defined, respectively, as follows:
\[
[A]_{\mathcal{H}} \doteq  \frac{1}{2}[(a + d) + i(c-b)] \quad \text{ and } \quad [A]_{\overline{\mathcal{H}}} \doteq  \frac{1}{2}[(a - d) + i(c+ b)]
\]
In particular, we have the fundamental identity:
\begin{equation}\label{e:fund}
	Az = [A]_{\mathcal H}z + [A]_{\overline{\mathcal H}}\overline{z}, \quad \forall z \in \R^2,
\end{equation}
where the product on the left-hand side is the classical one between a matrix and a vector and the one on the right-hand side is the complex multiplication. We have the following relations:
\begin{equation}\label{alg}
	\det(A) = |[A]_{\mathcal{H}}|^2 - |[A]_{\overline{\mathcal{H}}}|^2, \quad |A|^2 = 2|[A]_{\mathcal{H}}|^2 + 2|[A]_{\overline{\mathcal{H}}}|^2.
\end{equation}
As a short-hand notation, we will write expressions such as:
\begin{equation}\label{e:shn}
	A = (a,b),
\end{equation}
which mean that $a$ is the conformal part of $A$ and $b$ is the anti-conformal part of $A$.\ Finally, if $\Omega \subset \mathbb{C}$ is an open set and $f:\Omega \to \C$, whenever its differential makes sense we can set
\begin{equation}\label{wirt}
	\begin{split}
		\partial_zf &= f_z \doteq [Df]_{\mathcal{H}} = \frac{1}{2}[(\partial_1f_1 + \partial_2f_2) + i(\partial_1f_2 - \partial_2f_1)], \\
		\partial_{\bar z }f &= f_{\overline z} \doteq [Df]_{\overline{\mathcal{H}}} = \frac{1}{2}[(\partial_1f_1 - \partial_2f_2) + i(\partial_1f_2 + \partial_2f_1)].
	\end{split}
\end{equation}

\subsubsection{Essential range and generalized differentials}\label{sec:ergd}

As defined in \eqref{e:egd}, given any $w \in \Lip(\Omega,\R^n)$, $\Omega\subset \R^m$ open, we let $[Dw]$ be its essential range over $\Omega$, defined as the smallest closed set $K'$ such that $Dw(x) \in K'$ for a.e. $x \in \Omega$.\ It is immediate to see that if $E \subset \Omega$ is the set of Lebesgue points for $Dw$, then,
\begin{align*}
	[Dw] = \overline{\{Dw(x) \in \R^{n\times m}: x \in E\}}.
\end{align*}
This, together with \eqref{e:point}, is closely related to the set of \emph{reachable gradients of $w$}, see for instance \cite[Definition 3.1.10]{CS04}, which is typically defined by considering the set $E$ of differentiability points of $w$. 
In our setting, namely for solutions $w$ to \eqref{e:di}, there is no difference between the two sets: if $K$ fulfills \eqref{e:diffinc} and $w$ solves \eqref{e:di}, then the set of differentiability points for $w$ and the set of Lebesgue points for $Dw$ coincide.

\subsection{Three equivalent viewpoints}\label{sec:nbs}

We have the following equivalence:

\begin{prop}\label{p:11}
	Let $u \in \Lip(B_1)$ be a solution to \eqref{e:G}, where $G$ is defined on a compact set $F \supseteq [Du]$.\ Then, there exists $v \in \Lip(B_1)$ such that $w \doteq (u,v)$ solves
	\[
	D w \in K = \left\{\left(\begin{array}{cc}x & y \\ -G_2(x,y) & G_1(x,y)\end{array}\right): (x,y) \in F \right\}, \quad \text{ a.e. in }B_1.
	\]
	Moreover, $G$ satisfies \eqref{e:monoG} for all $(x,y) \in F \supseteq [Du]$ for a function $\sigma$ 
	fulfilling \eqref{e:sigmaprop} if and only if $K$ fulfills \eqref{e:diffinc} for some $\tilde{\sigma}$ fulfilling \eqref{e:sigmaprop}.\ Conversely, if $w = (u,v) \in \Lip(B_1,\R^2)$ satisfies \eqref{e:di} and $K$ is a compact set fulfilling \eqref{e:diffinc}, then there exists a field $G \in C^0(\pi_1(K), \R^2)$ fulfilling \eqref{e:monoG} such that $u$ solves \eqref{e:G}.\ Here, $\pi_1$ is the projection onto the first row.
\end{prop}
\begin{proof}
	To show the first statement we simply use \eqref{e:G} and the convexity of $B_1$ to find $v \in \Lip(B_1)$ such that $Dv = -JG(Du)$ a.e. in $B_1$. The rest of the properties are straightforward. Here and in what follows,
	\begin{equation}\label{e:J}
		J \doteq \left(\begin{array}{cc}0 & 1 \\ -1 & 0\end{array}\right).
	\end{equation}
	For the converse statement, we notice that inequality \eqref{e:diffinc} reads as
	\begin{equation}\label{e:XY}
		(JX_2 - JY_2,X_1 - Y_1) \ge \sigma(|X-Y|), \quad \forall X,Y \in K,
	\end{equation}
	if $X_i,Y_i$ are the $i$-th row of $X$ and $Y$ respectively.\ 
	Through \eqref{e:sigmaprop} we infer that the map $Z \in K \mapsto Z_1 \in \R^2$ is injective.\ Therefore, it must be invertible when the target is restricted to $\pi_1(K)$.\ Let $I(Z_1)$ be such inverse, and define also $\tilde G(Z_1) \doteq \pi_2(I(Z_1))$.\ Finally, set $G(x,y) \doteq J\tilde G(x,y)$, for all $(x,y) \in \pi_1(K)$.\ Writing \eqref{e:XY} for $X_1 = a, Y_1 = b \in \pi_1(K)$, we deduce that
		\[
		(G(b)-G(a),b-a) \ge \sigma(\sqrt{|b-a|^2 + |G(b) - G(a)|^2}).
		\]
		Observing that $G$ is bounded since $K$ is, \eqref{e:sigmaprop} implies that $G$ is continuous on $\pi_1(K)$ and that $G$ satisfies \eqref{e:monoG}.\ Moreover, by definition, $G(Du) = J\tilde G(Du) = JDv$ a.e., so $u$ solves \eqref{e:G}, and the proof is concluded.
\end{proof}

In the same spirit, as observed in \cite[Theorem~3.2]{zhang97}, we have:

\begin{prop}\label{p:12}
	Let $w \in \Lip(B_1,\R^2)$ be a solution to
	\begin{equation}\label{eq:g}
		\partial_{\bar z}w = h(\partial_z w) \text{ a.e. in }B_1,
	\end{equation}
	for $h\colon F \supseteq [\partial_zw] \to \mathbb{C}$, 
	 and $F$ a compact set.\ Using notation \eqref{e:shn}, if $K \doteq \left\{(a,h(a)): a\in F\right\}$, then $w$ solves \eqref{e:di} for such $K$.\ Moreover, if $h$ is strictly contractive, in the sense that there exists 
	\begin{equation}\label{e:sigmatilde}
		\text{ an increasing function $\tilde\sigma\colon [0,\infty)\to [0,1]$}
	\end{equation}
	such that
	\begin{equation}\label{e:contrac}
		|h(a)-h(b)|\leq (1-\tilde\sigma(|a-b|))|a-b|,\qquad\forall a,b\in F\,,
	\end{equation}
	then $K$ fulfills \eqref{e:diffinc} for some function $\sigma$ 
	with properties \eqref{e:sigmaprop}.\ Conversely, let $w \in \Lip(B_1,\R^2)$ solve \eqref{e:di}, and assume $K$ is a compact set fulfilling \eqref{e:diffinc}.\ Then, there exists a strictly contractive  $h\colon [K]_{\mathcal{H}} \to \mathbb{C}$ such that $w$ solves \eqref{eq:g}.\ We denoted by $[K]_{\mathcal{H}} \doteq \{[X]_{\mathcal{H}}:X \in K\}$.
\end{prop}
\begin{proof}
	The fact that $w$ solves \eqref{e:di} for $K$ given by the graph of $h$ is immediate.\ Now \eqref{e:diffinc} follows from \eqref{e:sigmatilde}-\eqref{e:contrac} by direct computations that exploit \eqref{alg}.\ The converse statement is less direct, but it follows from the same arguments as the previous proposition: we rewrite the ellipticity \eqref{e:diffinc} using \eqref{alg} as
	\begin{align*}
		\sigma(|X-Y|)+
		\big|[X]_{\overline{\mathcal H}}-[Y]_{\overline{\mathcal H}}\big|^2 
		\leq  
		\big|[X]_{\mathcal H}-[Y]_{\mathcal H}\big|^2
		\qquad\forall X, Y\in K\,.
	\end{align*}
	Then, this inequality implies that the projection  map $g\colon K\to \C$, $X\mapsto [X]_{\mathcal H}$ is injective and that the map $h\colon [K]_{\mathcal{H}} \to\mathbb C$, $z\mapsto [g^{-1}(z)]_{\overline{\mathcal H}}$, is strictly contractive.
\end{proof}

The last viewpoint we introduced in Proposition \ref{p:12}, the one of nonlinear Beltrami systems, is another rather useful one, see \cite{FS08,Astala2008} and references therein.\ We exploit it now to show how all of our problems \eqref{e:G}-\eqref{e:di} and \eqref{eq:g} can be \emph{extended} in a suitable way so that, for instance, instead of working with $G$ defined on $[Du]$ in \eqref{e:G} we can work with $G \in C^0(\R^2,\R^2)$ fulfilling some special properties \emph{at infinity}.\ Thanks to Propositions \ref{p:11}-\ref{p:12}, it is enough to present this extension for nonlinear Beltrami systems.

\begin{lem}\label{l:extension}
	Let $F\subset\C$ be compact and $h\colon F\to\C$ be a strictly contractive mapping, that is, $h$ satisfies \eqref{e:sigmatilde}-\eqref{e:contrac}.\ Then $h$ admits a strictly contractive extension $H\colon\C\to\C$, for a possibly different $\tilde \sigma$ than $h$ still fulfilling \eqref{e:sigmatilde}.\ Moreover, $H$ can be chosen to be constant outside a large ball.
\end{lem}
\begin{proof}
	The existence of a contractive extension $H_1$ of $h$ follows from \cite[Theorem~3.1]{DBP04}. Then it suffices to set $H=H_1\circ\Phi$, where $\Phi\colon\C\to\C$ is 1-Lipschitz, is the identity on $B_R$ for some $R>0$ such that $F\subset B_R$, and has compact support.
	To obtain such a map $\Phi$, one can set for instance
	$\Phi(z)=\chi(|z|)z$ with 
	\begin{align*}
		\chi(r)=\begin{cases}
			1 &\text{for }0\leq r\leq R\,,\\
			1-\ln\frac{r}{R} &\text{for }R<r<eR\,,\\
			0 &\text{for }R\geq eR\,.
		\end{cases}
	\end{align*}
	The differential of $\Phi$ at $z=re^{i\theta}$ is symmetric with eigenvalues $\chi(r)\in [0,1]$ and $\chi(r)+r\chi'(r)\in [-1,1]$.
	It has therefore operator norm at most $1$, hence $\Phi$ is $1$-Lipschitz.
\end{proof}

\section{Elliptic inclusions and separation}\label{sec:sep}

This section is devoted to showing the following result, for which it is useful to recall definition \eqref{e:point}.

\begin{thm}\label{t:noninc}
	Let $K \subset \R^{2\times 2}$ be compact and fulfill \eqref{e:diffinc}, $(w_j)_j\subset \Lip(B_1,\R^2)$ be equi-Lipschitz solutions to \eqref{e:di}, $x_j \in B_1$, and assume $w_j$ converges locally uniformly in $B_1$ to $w_\infty$, $x_j \to x_0 \in B_1$.\ Then, 
	\begin{equation}\label{eq:usc}
	\limsup_j\diam([\partial_zw_j](x_j)) \le \diam([\partial_zw_\infty](x_0)).
	\end{equation}
As a consequence,
	\begin{equation}\label{eq:usc1}
	\limsup_j\diam([Dw_j](x_j)) \le 2\diam([Dw_\infty](x_0)).
	\end{equation}
\end{thm}

\begin{rem}\label{rmk:simp}
Notice that we only need to show \eqref{eq:usc}, as \eqref{eq:usc1} follows from the fact that $w_j$ and $w$ solve \eqref{e:di} and $K$ is the graph of a $1$-Lipschitz maps over the plane of conformal matrices, compare Proposition \ref{p:12}.
\end{rem}

We reduce the proof of Theorem \ref{t:noninc} to the following:

\begin{thm}\label{t:locC}
Let $K\subset\R^{2\times 2}$ be compact and fulfill \eqref{e:diffinc}.
Let $w_0\in \Lip(B_1,\R^2)$ solve \eqref{e:di}, 
and 
$C\subset\R^2$ an open, bounded, convex set such that $[\partial_z w_0](B_1)\subset C$.
Then there exists $\delta=\delta(w_0,K,C)>0$  such that
\begin{align*}
\|w-w_0\|_{L^\infty(B_1)} \leq \delta
\quad
\Longrightarrow
\quad 
[\partial_z w](B_\delta)\subset \overline C\,,
\end{align*}
for all $w\in \Lip(B_1,\R^2)$ solving \eqref{e:di}.
\end{thm}

\begin{proof}[Proof of Theorem~\ref{t:noninc} from Theorem~\ref{t:locC}]
Let $\e>0$ and $C\subset\R^2$ be an open convex set containing $[\partial_z w_\infty](x_0)$ with 
$\diam(C)\leq \diam([\partial_z w_\infty](x_0))+\e$.
By definition \eqref{e:point} of the essential range at $x_0$, there exists $r>0$ such that $B_r(x_0)\subset B_1$ and $[\partial_z w_\infty](B_r(x_0))\subset C$.
Applying Theorem~\ref{t:locC} to the rescaled map
 $w_0(x)=(w_\infty(x_0+rx)-w_\infty(x_0))/r$,
 and scaling back,
  we deduce the existence of $\delta \in (0,r)$ such that
$[\partial_z w](B_\delta(x_0))\subset\overline C$,
  for any $w$ solving \eqref{e:di} and such that $\|w-w_\infty\|_{L^\infty(B_1)}\leq \delta$.
Fixing $J\geq 0$ such that
 $\|w_j-w_\infty\|_{L^\infty(B_1)}\leq\delta$ and $|x_j-x_0|\leq\delta/2$ for all $j\geq J$,
we have therefore
 \begin{align*}
 [\partial_z w_j](x_j)\subset [\partial_z w_j](B_\delta(x_0))\subset\overline C\,,
 \end{align*}
 hence $\limsup_j\diam( [\partial_z w_j](x_j))\leq \diam(\overline C)=\diam(C)\leq \diam([\partial_z w_\infty](0))+\e$.
 Letting $\e\to 0$ concludes 
the proof of \eqref{eq:usc}.\ 
As said in Remark \ref{rmk:simp}, \eqref{eq:usc1} follows from it directly.
\end{proof}

We will show Theorem \ref{t:locC} in \textsection \ref{sec:proof}, and we start by collecting a few preliminary results.

\subsection{Preliminary results}

\subsubsection{Topological degree}

In this section we recall some well-known results on the topological degree $\deg(u,\Omega,p)$.\ We refer the reader to \cite{Fonseca1995a} and references therein for the definition and a detailed introduction.

\begin{prop}
	Let $\Omega \subset \R^n$ be an open, bounded and connected set and let $u \in W^{1,p}(\Omega,\R^n)$ for $p > n$. Then for all open sets $U \Subset \Omega$ with $|\partial U| = 0$:
	\begin{equation}\label{e:area}
		\int_U v(u(x))\det(Du)(x)dx = \int_{\R^n}v(y)\deg(u,U,y)dy, \quad \forall v \in L^\infty(\R^n),
	\end{equation}
	and, if $\det(Du) > 0$ a.e. in $\Omega$, then
	\begin{equation}\label{e:degN}
		\deg(u,U,y) = N(u,U,y), \quad \text{ for a.e. }y \in \R^n,
	\end{equation}
	where $N(u,U,y) \doteq \#\{x \in U: u(x) = y\}$.
\end{prop}
\begin{proof}
 \eqref{e:area} is shown in \cite[Theorem 5.31]{Fonseca1995a}. Through \eqref{e:area} and \cite[Theorem 5.30]{Fonseca1995a}, we also get \eqref{e:degN}.
\end{proof}

\begin{prop}\label{prop:deg}
	Let $\Omega \subset \R^n$ be open and bounded, $v \in C^0(\overline{\Omega}, \R^n)$, and let $p  \in \R^n\setminus v(\partial \Omega)$.\ Then:
	\begin{equation}\label{e:uv}
		\text{$\deg(v,\Omega,p) = \deg(u,\Omega,p)$ if $u \in C(\overline{\Omega}, \R^n)$ and $\|u-v\|_{\infty} < \dist(p,v(\partial \Omega))$};
	\end{equation}
	Moreover, the degree is invariant under homotopies, namely
	\begin{equation}\label{eq:homoto}
		\deg(H(\cdot,t),\Omega,p) = \deg(H(\cdot,0),\Omega,p),
	\end{equation} 
	for every homotopy $H \in C^0(\overline{\Omega}\times [0,1], \R^n)$ such that $p \notin H(\partial \Omega,t)$ for all $t \in [0,1]$.
\end{prop}
\begin{proof}
This is classical and can be found e.g. in  \cite[Theorem~2.3]{Fonseca1995a}.
\end{proof}

\begin{rem}\label{rem:deg}
	The degree $\deg(u, \Omega, p)$ is, in general, not defined if $p \in u(\partial \Omega)$.\ Therefore, every time we write expressions involving $\deg(u,\Omega,p)$ and $p \in B$, for some set $B$, this will implicitly entail that $u(\partial\Omega) \cap B = \emptyset$.
\end{rem}

\subsubsection{Quasiregular mappings}\label{sec:qr}

We start by recalling the definition.\
\begin{Def}\label{def:qua}
	Let $\Omega \subset \R^2$ be open. $\varphi\in W^{1,2}(\Omega,\R^2)$ is ($K$-)quasiregular if there exists $K \ge 1$ such that
	\begin{equation}\label{Kqc1}
		|D\varphi|^2(x) \le K\det(D\varphi(x)), \quad \text{for a.e. $x \in \Omega$.}
	\end{equation}
	Equivalently, $f \in W^{1,2}_{\loc}(\Omega,\R^2)$ is quasiregular if and only if there exists $\kappa \in [0,1)$ such that
	\begin{equation}\label{fqccomp}
		|f_{\overline z}|(z) \le \kappa|f_{z}|(z) \text{ for a.e. $z \in \Omega$}.
	\end{equation}
\end{Def}

\begin{rem}\label{r:diff_quasi}
	For instance, if we consider equation \eqref{eq:g} for some $k$-Lipschitz map $h$ for some $0\leq k<1$, then the difference $u=v-w$ of any two solutions $v,w$ to \eqref{eq:g} is quasiregular:
	\begin{align*}
		|u_{\bar z}|^2 =|v_{\bar z}-w_{\bar z}|^2 = |h(v_z)-h(w_z)|^2 \leq k^2 |v_z-w_z|^2=k^2 |u_z|^2\,.
	\end{align*}
\end{rem}

Let us recall a few properties of quasiregular maps, referring the reader to \cite{Bojarski1983,Astala2008,Iwaniec2001} for more details.

\begin{prop}\label{thm:quas}
	Let $\varphi \in W^{1,2}(\Omega,\C)$ be quasiregular. Then:
	\begin{enumerate}
		\item\label{quasma0} there exists $p = p(K) > 2$ such that $\varphi \in W^{1,p}_{\loc}(\Omega,\R^2)$. In particular, $\varphi$ is continuous;
		\item\label{quasma1} $\varphi$ is either constant, or open and discrete;
		\item\label{quasma2} $\varphi$ is either constant, or $\det(D\varphi) > 0$ a.e. in $\Omega$;
		\item\label{quasma3} if $U,W \subset \R^2$ are open sets and $N(\varphi,U,y) = 1$ for a.e. $y \in W$, then $\varphi$ is injective on $U \cap \varphi^{-1}(W)$.
	\end{enumerate}
\end{prop}
\begin{proof}
	\eqref{quasma0} can be found in \cite[Theorem 5.1]{Bojarski1983} (see also \cite{Astala1994} for the precise expression of $p$). \eqref{quasma1}-\eqref{quasma2} can be found in \cite[Corollary 5.5.2]{Astala2008}. Let us show \eqref{quasma3}, following the same argument of \cite[Lemma 4.3]{DPGT24}. Observe that $\varphi$ is open by \eqref{quasma1}. Suppose by contradiction that there exist distinct $x_1,x_2 \in U \cap \varphi^{-1}(W)$ such that $\varphi(x_1) = \varphi(x_2) = y$.\ Then, taking a sufficiently small $r > 0$ such that 
	\begin{equation}\label{e:disj}
		B_r(x_1)\cap B_r(x_2) = \emptyset
	\end{equation}
	and $B_r(x_i) \subset U \cap \varphi^{-1}(W)$ for all $i$, define the open set $V = \varphi(B_r(x_1))\cap \varphi(B_r(x_2)) \subset W$.\ Observe that $y \in V$. Since $V$ is open and nonempty, our assumption implies that we can find $p \in V$ for which $N(\varphi,U,p) = 1$. This yields a contradiction with the definition of $V$ and \eqref{e:disj}.
\end{proof}

We state the \emph{stability} result for quasiregular maps established in \cite[Proposition 1]{KS08}.\ In that reference, the statement is shown in $\R^n$, but we will only need it for planar maps.\ We derive from it Corollary \ref{c:stab}, which is essentially contained in \cite[Proposition~2]{KS08}.

\begin{prop}\label{p:stab}
	Let $u:\Omega \subset \R^2 \to \R^2$ be a $K$-quasiregular mapping such that $|Du(x)| \ge \eps$ a.e. in $\Omega$.\ Let $G \Subset \Omega$ and assume that $M\doteq \sup_{y \in \R^2}N(u,G,y)< +\infty$. Then there exist a constant $\delta = \delta(\eps,K,M)>0$ and for any $x_0 \in G$ a radius $r(x_0)> 0$ such that for any Lipschitz mapping $\phi : \Omega \to \R^2$ with $\|D\phi\|_\infty < \delta$,
	\[
	\min_{|x-x_0|= r}|u^t(x)-u^t(x_0)| \ge \delta r, \quad \text{ for all } r < r(x_0), t \in [0,1], \text{ if }u^t = u + t\phi.
	\]
\end{prop}

\begin{rem}\label{r:M}
There are several ways to see that the assumption $M<+\infty$ in the previous proposition is automatically satisfied.
 One can use, for instance, Stoilow Factorization \cite[Theorem 5.5.1]{Astala2008} or more purely topological methods valid in all dimensions \cite[Proposition 4.10(3)]{Rickman1993}.
\end{rem}

\begin{cor}\label{c:stab}
Let $R,\e>0$, $u\in \Lip(B_{R},\R^2)$ such that $u(0)=0$, and $\Gamma\subset\R^{2\times 2}$ be a compact and path-connected set of matrices such that
\begin{equation}\label{e:BELOW}
\det(Du-A)\geq\e\quad\text{a.e. in }B_R\,, \quad\forall A\in\Gamma\,.
\end{equation}
Then there exists $r_0=r_0(u,\Gamma)\in (0,R)$ 
and $\delta=\delta(\e,\Gamma)>0$ such that
\begin{align*}
A \mapsto \deg(u-L_A,B_r,y)
\text{ is defined and constant on }\Gamma 
\text{ for all }r\in (0,r_0) \text{ and } y \in \overline{B_{\delta r}}\,.
\end{align*} 
\end{cor}
\begin{proof}
Fix $\Lambda>0$ such that $|Du|+|A|\leq \Lambda$ for all $A\in\Gamma$ and a.e. $x \in B_R$.\ The finiteness of such $\Lambda$ and \eqref{e:BELOW} imply that the map $u^A\doteq u-L_A$ is a nonconstant quasiregular mapping in $B_{R}$ for all $A\in\Gamma$. 
Thanks to Remark~\ref{r:M} we may apply Proposition~\ref{p:stab} 
to $u^A$ at $x_0 = 0$. 
Recalling that $u^{A'}(0) = 0$ 
for all $A' \in \Gamma$, this yields the existence of $r_A =r_A(u)\in (0,R)$ and $\delta_A=\delta_A(\e,\Lambda) >0$ 
such that
\begin{align*}
\dist(0,u^{A'}(\partial B_r))>\delta_A r\quad\forall r\in (0,r_A),\; \forall A'\in B_{\delta_A}(A)\,.
\end{align*}
Covering $\Gamma$ with a finite number of open balls $B_{\delta_A}(A)$,
 we infer the existence of $r_0>0$ and $\delta>0$ such that
\begin{align*}
\dist(0,u^{A}(\partial B_r))>\delta r\quad\forall r\in (0,r_0),\; \forall A\in\Gamma\,.
\end{align*}
Thus for all $A\in\Gamma$ we have
$u^A(\partial B_r)\cap \overline{B_{\delta r}}=\emptyset$,
hence the degree
$\deg(u^A,B_r,y)$ is well-defined for all $y\in \overline {B_{\delta r}}$.
Moreover, it is independent of $A\in\Gamma$ by path-connectedness of $\Gamma$ and homotopy invariance of the degree, see Proposition~\ref{prop:deg}.
\end{proof}

\subsubsection{Approximation of the problem}

Thanks to Proposition~\ref{p:12} and Lemma~\ref{l:extension} we may
 write 
 $K\subset\hat K$,
 where $\hat K$ still fulfills \eqref{e:diffinc} and  
 is the graph of $h\colon\C\to\C$  satisfying
\begin{align}
	&|h(a)-h(b)| \le (1-\tilde\sigma(|a-b|))|a-b|, \quad \text{for all }a,b \in \C, \label{eq:h1}\\
	&\|h\|_{L^\infty(\C,\C)} < + \infty, \label{eq:h2}
\end{align}
for some $\tilde\sigma$ fulfilling \eqref{e:sigmatilde}.
\ 
For a Lipschitz map $w$ solving \eqref{e:di} we have therefore
\begin{equation}\label{eq:w}
	w_{\bar z} = h(w_z) \text{ a.e. in }B_1.
\end{equation}
Let $w$ be a solution to \eqref{eq:w} and $h$ satisfy \eqref{eq:h1}-\eqref{eq:h2}.\ We want to obtain $w$ as limit of solutions to strongly elliptic problems.
 
\begin{lem}\label{lem:app}
	For every $\eps \in (0,1)$, the nonlinear Beltrami system
	\begin{equation}\label{eq:eps}
		\begin{cases}
			\partial_{\bar z}w^\eps = (1-\eps)h(\partial_zw^\eps) \text{ in }B_1,\\
			\Ree(w^\eps) = \Ree(w) \text{ on }\partial B_1.
		\end{cases}
	\end{equation}
	admits a unique solution $w^\eps \in W^{1,2}(B_1,\C)$ with $w^\eps_2$ of zero average. Furthermore, $w_\eps \to w$ locally uniformly and locally strongly in $W^{1,p}$ for all $p\in [1,\infty)$. 
\end{lem}
\begin{proof}
	The solvability of this system can be found in \cite[Proposition 2]{FS08}, where it is shown that 
if $H: B_1\times \mathbb{C} \to \mathbb{C}$ is a measurable function satisfying $H(z,0) = 0$ for a.e. $z \in B_1$ and
	\[
	|H(z,w_1)-H(z,w_2)| \le k|w_1-w_2|, \text{ for a.e. $z \in B_1$ and all $w_1,w_2\in \mathbb{C}$ for some $k < 1$},
	\]
	then for any $\sigma \in L^2(B_1,\mathbb{C})$ there exists a $W^{1,2}(B_1,\mathbb{C})$ solution $v$ (which is unique up to the addition of a constant to the second component) of 
	\begin{equation}\label{eq:epsv}
		\begin{cases}
			\partial_{\bar z}v = H(z,\partial_zv) + \sigma(z) \text{ in }B_1,\\
			\Ree(v) = 0 \text{ on }\partial B_1.
		\end{cases}
	\end{equation}
	To solve \eqref{eq:eps} we can just consider a solution $v^\eps$ to \eqref{eq:epsv} with $$H_\eps(z,\xi) \doteq (1-\eps)(h(\partial_zw + \xi)-h(\partial_zw)) \text{ and } \sigma(z) \doteq - \eps h(\partial_zw).$$ and then set $w^\eps \doteq v^\eps + w$.\ We will now show the required estimates.
 We start by obtaining $W^{1,2}$ bounds uniform in $\eps > 0$.\ Using complex notation, see \eqref{alg}:
	\begin{align*}
		\int_{B_1}|Dw^\eps|^2dx & = 2 \int_{B_1}|w^\eps_z|^2+ |w^\eps_{\bar z}|^2 dx = 2\int_{B_1}|w^\eps_z|^2- |w^\eps_{\bar z}|^2 dx + 4\int_{B_1}|w^\eps_{\bar z}|^2dx.
	\end{align*}
	The last addendum is bounded, since our assumptions on $h$, see \eqref{eq:h2}, and \eqref{eq:eps} imply that actually $w_{\bar z}^\eps$ is equibounded in $L^\infty$. It then suffices to estimate the first addendum using the boundary conditions:
	\[
	\int_{B_1}|w^\eps_z|^2- |w^\eps_{\bar z}|^2 dx \overset{\eqref{alg}}{=} \int_{B_1} \det(Dw^\eps)dx = \int_{B_1}\det(D\psi^\eps)dx,
	\]
	where $\psi^\eps = (\Ree(w), w^\eps_2)$. Now by Young's inequality, for any $\delta > 0$ we find:
	\[
	\int_{B_1}\det(D\psi^\eps)dx \le C(\delta)\int_{B_1}|Dw|^2dx + \delta \int_{B_1}|Dw^\eps|^2dx.
	\]
	Combining all these estimates, we get a universal constant $C$ such that
	\[
	\int_{B_1}|Dw^\eps|^2dx \le C\bigg(C(\delta)\int_{B_1}|Dw|^2dx + \delta \int_{B_1}|Dw^\eps|^2dx + C(h)\bigg).
	\]
By choosing $\delta$ small, we deduce uniform $L^2$ bounds on $Dw_\eps$.\ Through Poincaré inequality, our boundary conditions and the fact that $w_2^\eps$ has zero average, we get that $w_\eps$ is equibounded in $L^p$ for any $p \in [1,+\infty)$, since $W^{1,2}(B_1)\subset L^p(B_1)$. Now we can use the boundedness of the Beurling transform \cite[Theorem~4.5.3]{Astala2008} to obtain interior $W^{1,p}$ bounds for all $p \ge 1$, uniformly in $\eps >0$: indeed, for any $\eta\in C_c^\infty(B_1)$, $(\eta w^\e)_{\bar z}=\eta_{\bar z} w^\e +(1-\e) \eta h(w^\e_z)$ is uniformly bounded in $L^p$ since $w^\eps$ is and $h\in L^\infty(\C)$, and thus $D(\eta w_\e)$ is uniformly bounded in $L^p$.\ We only need to show the convergence of $w^\eps$ to $w$. Using the boundary conditions,
	\begin{align*}
		0 = \int_{B_1}\det(Dw^\eps - Dw)dx &= \int_{B_1}|\partial_zw^\eps - \partial_{z}w|^2 - |\partial_{\bar z}w^\eps - \partial_{\bar z}w|^2dx\\
		& = \int_{B_1}|\partial_zw^\eps - \partial_{z}w|^2 - |(1-\eps)h(\partial_zw^\eps) - h(\partial_z w)|^2dx
	\end{align*}
	Thanks to the uniform bounds shown above, this gives us for all $\eps> 0$:
	\[
	\int_{B_1}|\partial_zw^\eps - \partial_{z}w|^2 - |h(\partial_zw^\eps) - h(\partial_z w)|^2 dx \le C\eps.
	\]
	From \eqref{eq:h1} and \eqref{e:sigmatilde}, this readily yields the convergence in measure of $\partial_zw^\eps$ to $\partial_zw$ as $\eps \to 0$. In turn \eqref{eq:eps} tells us that $\partial_{\bar z}w^\eps$ does the same and this concludes the proof.
\end{proof}

\subsection{Proof of Theorem \ref{t:locC}}\label{sec:proof}

Recall that $K\subset\hat K$,
 where $\hat K$ fulfills \eqref{e:diffinc} and  
 is the graph of $h$  satisfying
 \eqref{eq:h1}-\eqref{eq:h2}-\eqref{eq:w}.\ Without loss of generality we assume $w(0)=w_0(0)=0$. For brevity, recalling \eqref{e:shn}, let
\[
g: \R^2 \to \hat K \text{ be defined as } g(z) \doteq (z,h(z)),  
\text{ and } \hat K_C \doteq g\left(\R^2\setminus C\right)\subset\hat K\,.
\]

\medskip

\noindent
\fbox{Step 1: \emph{degree of $w -L_A$ for $A\in \hat K_C$}.}
We show the existence of $\alpha>0$ and $0<r<1$ such that
\begin{align}\label{eq:deg_w}
\deg(w-L_A,B_r,y)=1\,,\quad\forall y\in\overline {B_{2\alpha}}\,,
\;\forall A\in \hat K_C\,,
\end{align}
provided $\|w-w_0\|_{L^\infty(B_1)}\leq \alpha$.
Recall 
that this assertion implicitly entails $(w-L_A)(\partial B_r)\cap \overline {B_{2\alpha}} =\emptyset$,
see Remark~\ref{rem:deg}.
Thanks to  Proposition~\ref{prop:deg}, it suffices to show
\begin{align}\label{eq:deg_winfty}
\quad\deg(w_0-L_A,B_r,y)=1\,,\quad\forall y\in\overline {B_{4\alpha}}\,,
\;\forall A\in \hat K_C\,.
\end{align}
From \eqref{e:diffinc} and the facts that $Dw_0\in  K$ a.e., 
$\hat K_C\subset\hat K$ and $\hat K_C \cap [Dw_0](B_1)=\emptyset$, we see that
\begin{align*}
\det(Dw_0-A)\geq \e\quad\text{a.e. in }B_1\,,\quad\forall A\in \hat K_C\,,
\end{align*}
where $\e=\sigma(\dist(\hat K_C,[Dw_0](B_1)))>0$.
Noting that the set $\hat K_C\cap g(\overline{B_M})$ is compact and path-connected for any $M>0$ such that $\overline C\subset B_M$, 
we apply Corollary~\ref{c:stab} and deduce the existence of 
$r=r(M)\in (0,1)$, $\delta=\delta(M)>0$ such that
\begin{align*}
A\mapsto \deg(w_0 -L_A,B_{r},y)
\text{ is defined and constant on }\hat K_C \cap g(\overline{B_M})\,,
\text{ for all }y\in \overline B_{\delta}\,.
\end{align*}
To conclude the proof of  \eqref{eq:deg_winfty}, 
it suffices therefore to find $M>0$ such that
\begin{align}\label{eq:deg_winfty_M}
\quad\deg(w_0-L_A,B_{\rho},y)=1\quad\forall y\in\overline{B_{\rho}}\,,
\;\forall \rho\in (0,1)\,,
\;\forall A\in g(B_M^c)\,,
\end{align}
apply this to $\rho=r$ and choose $4\alpha=\min(\delta,r)$, for $r,\delta$ given by this choice of $M$.\ 
To establish \eqref{eq:deg_winfty_M}, let $M\geq 2\|h\|_\infty$, $z\in B_M^c$, so that $|z| \ge M \ge 2\|h\|_\infty$, and $A=g(z)=(z,h(z))$.
Recalling \eqref{alg}, we have 
\begin{align*}
|z|^2 \leq |A|^2=|z|^2+|h(z)|^2 \leq 2|z|^2\,,
\quad
\frac{1}{2}|z|^2 \leq \det(A)=|z|^2-|h(z)|^2 \leq |z|^2\,.
\end{align*}
This implies 
 that $A$ is invertible and 
 $|A^{-1}|=(\det A)^{-1}|A|\ < 4/M$, so
 $M|x| < 8|Ax|$ for all $x\in \R^2\setminus\{0\}$.
Hence we have $L_{A}(\partial B_\rho)\cap \overline{B_{M\rho/8}}=\emptyset$ for all $\rho>0$,
and
\begin{align*}
\deg(-L_{A},B_\rho,y)=1\,,\quad\forall y\in \overline{B_{M\rho/8}}\,.
\end{align*}
Assume moreover $M>16\Lip(w_0)$.\ Then for $0<\rho<1$ we have $\|w_0\|_{L^\infty(B_\rho)}<M \rho/16$ and we can invoke Proposition~\ref{prop:deg} to deduce
\begin{align*}
\deg( w_0 -L_A ,B_\rho,y)=1\,,\quad
\forall y\in \overline{B_{M\rho/16}}\,,\;\forall \rho\in (0,1)\,,
\end{align*}
which gives \eqref{eq:deg_winfty_M} if we impose in addition $M\geq 16$. This proves Step 1.

\medskip

From now on we fix $w\in \Lip(B_1,\R^2)$ solving \eqref{e:di} and such that $\|w-w_0\|_{L^\infty(B_1)}\leq\alpha$, so that \eqref{eq:deg_w} is satisfied.
To conclude the proof of Theorem~\ref{t:locC}, 
the main idea is that
the degree property \eqref{eq:deg_w} is not so far from implying that $w-L_A$ is injective in the preimage of $B_\alpha$, because $w-L_A$ is orientation preserving thanks to \eqref{e:diffinc}.
This injectivity, for all $A\in\hat K_C$, would then imply strong constraints on the difference quotients of $w$
and eventually the desired conclusion that $Dw(x)\in \overline C$.
In fact, that \eqref{eq:deg_w} implies injectivity would be  true if $w-L_A$ were quasiregular, see Proposition~\ref{thm:quas}. Hence we apply the strategy we just described, but at the level of the approximation $w^\e$ provided by Lemma~\ref{lem:app}.

\medskip

\noindent
\fbox{Step 2: \emph{Approximation}.}
We start by fixing $R>0$ such that   $\overline C\subset B_R$ and $K\subset g( B_R)$.
In particular we have $\Lip(w)\leq M\doteq R+\|h\|_\infty$.
Recalling also that $w(0)=0$ we infer,
for $\eta\doteq\min(r/2,\alpha/(8M))>0$,
\begin{align}\label{eq:Beta}
B_{2\eta}\subset B_r\cap (w-L_A)^{-1}(\overline{B_{\alpha/2}})\,,
\quad\forall A\in \hat K_C\cap g(\overline{B_R})\,.
\end{align}
Next we fix $A\in \hat K_C\cap g(\overline{B_R})$, 
and  $a\in\overline{B_R}\setminus C$ such that
$A=g(a)=(a,h(a))$.
For $0<\e<1$ 
we let $w^\e$ be given by Lemma~\ref{lem:app}, and $A^\e =(a,(1-\e)h(a))$,
so that the maps $w^\e-L_{A^\e}$ are quasiregular,
see Remark~\ref{r:diff_quasi}.
Since $w^\e\to w$ uniformly in $\overline{B_r}$
and $|A^\e-A|\leq\e \|h\|_\infty$,
we may 
choose $\e_0>0$ such that
\begin{align*}
\|(w^\e-L_{A^\e})-(w-L_A)\|_{L^\infty(B_r)}< \frac\alpha 2 \,,
\quad\forall \e\in (0,\e_0)\,.
\end{align*}
Note that $\eps_0$ is independent of $A$. 
Next, we fix $\e\in (0,\e_0)$ and combine this with
 \eqref{eq:Beta}
to obtain
\begin{align}\label{eq:Betaeps}
B_{2\eta}\subset B_r\cap (w^\e-L_{A^\e})^{-1}(\overline{B_{\alpha}})\,.
\end{align}
Combining it also with \eqref{eq:deg_w}, which, recall Remark~\ref{rem:deg}, entails $(w-L_A)(\partial B_r)\cap \overline{B_{2\alpha}}=\emptyset$, gives
\begin{align}\label{eq:degweps}
\deg(w^\e-L_{A^\e},B_r,y)=1\,,\quad\forall y\in\overline{{B_{\alpha}}}\,,
\end{align}
thanks to Proposition~\ref{prop:deg}.
In particular, $w^\e-L_{A^\e}$ cannot be constant on $B_r$, hence it has positive Jacobian a.e. by Proposition~\ref{thm:quas}\eqref{quasma2}.
Recalling from Lemma~\ref{lem:app} that $w^\e\in W^{1,p}(B_r)$ for all $p\geq 1$,
we can apply \eqref{e:degN} and deduce
\begin{equation*}
N(w_{\e}-L_{A^\e}, B_{r},y) = \deg(w_\e-L_{A^\e}, B_{r},y)  = 1, \quad \text{for a.e. } y \in  \overline{B_{\alpha}}\,. 
\end{equation*}
Through Proposition \ref{thm:quas}\eqref{quasma3} and \eqref{eq:Betaeps}, we infer that
\begin{align}\label{eq:inj}
w^\e-L_{A^\e}
\text{ is injective on }B_{2\eta}\,,
\quad\forall \e\in (0,\e_0)\,,
\;\forall A^\e\in g^\e(\overline{B_R}\setminus C)\,,
\end{align}
where $g^\e(z)=(z,(1-\e)h(z))$.

\medskip

\noindent
\fbox{Step 3: \emph{Separation and conclusion.}}
For any $x\in B_\eta$ and $0<t<\eta$ 
we deduce 
from \eqref{eq:inj} that 
$x$ and $x+te_1$ have different images through $w^\e-L_{A^\e}$, which can be rewritten as
\begin{align}\label{eq:quotient_eps}
\frac{w^{\e}(x+te_1)-w^{\e}(x)}{t}\neq A^\e e_1\,,
\quad\forall \e\in (0,\e_0)\,,
\;\forall A^\e\in g^\e(\overline{B_R}\setminus C)\,,
\end{align}
For all $\e\in [0,1)$, the continuous map $\Phi^\e\colon\R^2\to\R^2$ given by
\begin{align*}
\Phi^\e(z)=g^\e(z)e_1=ze_1+(1-\e)h(z)e_1\,,
\end{align*}
is injective thanks to
 \eqref{eq:h1},
 and has closed image
 thanks to the boundedness of $h$ \eqref{eq:h2},
 hence $\Phi^\e$ is a homeomorphism of $\R^2$.
Since $C$ and $\R^2\setminus \overline{B_R}$ are disjoint connected sets,
we deduce that the complement of $\Phi^\e(\overline{B_R}\setminus C)$ has
two connected components, $\Phi^\e(\R^2\setminus \overline{B_R})$ and $\Phi^\e(C)$.
The image of the continuous map
\begin{align*}
B_\eta\times (0,\eta) \ni (x,t)\mapsto
\frac{w^{\e}(x+te)-w^{\e}(x)}{t}\,,
\end{align*}
lies therefore, by \eqref{eq:quotient_eps}, in one of these two connected components:
for any  $\e\in (0,\e_0)$ we have 
\begin{align*}
\text{either }
&
\frac{w^{\e}(x+te_1)-w^{\e}(x)}{t}\in \Phi^\e(C)\quad\forall (x,t)\in B_\eta\times (0,\eta)\,,
\\
\text{or }
&
\frac{w^{\e}(x+te_1)-w^{\e}(x)}{t}\in \Phi^\e(\R^2\setminus \overline{B_R})\quad\forall (x,t)\in B_\eta\times (0,\eta)\,.
\end{align*}
One of these must be satisfied for infinitely many $\e\in (0,\e_0)$,
and from the uniform convergences of $w^\e$ to $w$ and $\Phi^\e$ to $\Phi^0\colon z\mapsto g(z)e_1$ 
we deduce 
\begin{align*}
\text{either }
&
\frac{w(x+te_1)-w(x)}{t}\in \Phi^0(\overline{C})\quad\forall (x,t)\in B_\eta\times (0,\eta)\,,
\\
\text{or }
&
\frac{w(x+te_1)-w(x)}{t}\in \Phi^0(\R^2\setminus {B_R})\quad\forall (x,t)\in B_\eta\times (0,\eta)\,.
\end{align*}
Letting $t\to 0$, at every differentiability point $x\in B_\eta$ of $w$ we have therefore
\begin{align*}
\Phi^0(\partial_z w(x))= Dw(x)e_1\in \Phi^0(\overline C)\cup \Phi^0(\R^2\setminus B_R)\,,
\end{align*}
hence $\partial_z w(x)\in \overline C\cup (\R^2\setminus B_R)$, 
as $\Phi^0$ is a homeomorphism.
Combining this with the fact that $Dw\in K\subset g(B_R)$ a.e., which implies $\partial_z w\in B_R$ a.e., we eventually 
deduce that $\partial_z w(x)\in\overline C$ for a.e. $x\in B_\eta$.
Setting $\delta=\min(\alpha,\eta)>0$ concludes the proof of Theorem~\ref{t:locC}.
\qed 

\section{The size of non-differentiability points}\label{sec:size}

In this section we show the following result valid in $\R^n$.

\begin{thm}\label{p:sing}
	Let $u \in \Lip(B_1)$, $B_1 \subset \R^n$, be a solution to \eqref{e:G}, where $G \in C^0([Du],\R^n)$ fulfills
\begin{align}\label{e:monoG_Rn}
(G(b)-G(a),b-a)\geq \sigma(|b-a|)\qquad\forall a,b\in [Du]\,,
\end{align}
for some $\sigma$ as in
\eqref{e:sigmaprop}. Then, $\mathcal{H}^{n-1}(S) = 0$, where $S$ is the set
	\begin{equation}\label{e:Set}
	S \doteq \left\{x \in B_1: \liminf_{r\to 0}\fint_{B_r(x)} \sigma\left(\left|Du(y) - \fint_{B_r(x)}Du(z) dz\right|\right)dy > 0\right\}.
	\end{equation}
\end{thm}
\begin{proof}
As the statement is local, we can assume that $u$ is defined in $B_2$.\ We start with a preliminary estimate. Let $|h| < \frac{1}{2}$ and $\varphi$ be a smooth cut-off function of $B_{\frac{3}{2}}$ inside $B_2$.\ Then, \eqref{e:G} implies:
\[
\int_{B_2}(G(Du(x + h))-G(Du(x)),D[\varphi(x)(u(x + h)-u(x))])dx = 0.
\]
Routine calculations and \eqref{e:monoG_Rn} then yield
\begin{equation}\label{e:Dh}
	\begin{split}
\int_{B_{\frac{3}{2}}} \sigma(|Du(x + h) - Du(x)|)dx &\le \int_{B_2}\varphi \sigma(|Du(x + h) - Du(x)|)dx\\
 &\le \int_{B_2}|D\varphi||u(x + h)-u(x)||G(Du(x + h))-G(Du(x))|dx \\
&\le C|h|\int_{B_2}|D\varphi||G(Du(x + h))-G(Du(x))|dx \le C|h|D(|h|),
\end{split}
\end{equation}
for $C = \Lip(u)$ and $$D(t) \doteq \sup_{h:|h| \le t}\int_{B_2}|D\varphi||G(Du(x + h))-G(Du(x))|dx.$$ Observe that
\begin{equation}\label{eq:Dh}
	\lim_{t\to 0}D(t) = 0.
\end{equation}
We now turn to our main goal, i.e. showing that $\mathcal{H}^{n-1}(S) = 0$. 
We notice that 
\begin{equation}\label{e:setS}
S \subset \bigcup_{p \ge 1} \bigcup_{m \ge 10} S_{p,m} = \bigcup_{p \ge 1} \bigcup_{m \ge 10} \bigcap_{0 < r < \frac{1}{m}} E_{p,r},
\end{equation}
where
\[
E_{p,r} \doteq \left\{x \in B_1: \fint_{B_r(x)}\sigma\left(\left|Du(y) - \fint_{B_r(x)}Du(z) dz\right|\right)dy \ge \frac{1}{p}\right\}.
\] 
Let us recall that $\mathcal H^{n-1}(E)=\lim_{r\to 0}\mathcal H^{n-1}_r(E)$, where, for $r>0$,
\[
\mathcal{H}_{r}^{n-1}(E) \doteq  \inf\left\{\sum_{i = 1}^\infty\diam(C_i)^{n-1}: E \subset \bigcup_{i = 1}^\infty C_i, \diam(C_i) \le r\right\}
\]
see \cite[Definition 2.1]{Evans2015}.\  Our definition differs from that of \cite[Definition 2.1]{Evans2015} by a constant factor, which is anyway irrelevant for what we need to show.\ We claim that:
\begin{equation}\label{eq:ficlar}
\lim_{r \to 0}\mathcal{H}_{5r}^{n-1}(E_{p,r}) = 0.
\end{equation}
If this holds, then $\mathcal{H}^{n-1}(S) = 0$, since
\begin{equation}\label{eq:ineqS}
\mathcal{H}^{n-1}_{5r}(S_{p,m}) \le \mathcal{H}^{n-1}_{5r}(E_{p,r}), \text{ for all } r < \frac{1}{m},
\end{equation}
and again by \cite[Definition 2.1]{Evans2015}:
\[
0 \overset{\eqref{eq:ficlar}-\eqref{eq:ineqS}}{=} \lim_{r\to 0}\mathcal{H}^{n-1}_{5r}(S_{p,m}) = \mathcal{H}^{n-1}(S_{p,m}).
\]
The $\sigma$-subadditivity of $\mathcal{H}^{n-1}$, \cite[Theorem 2.1, Claim \# 2]{Evans2015}, and \eqref{e:setS} would then conclude the proof.\ We are only left to show \eqref{eq:ficlar}.\ To this aim, pick any $x \in E_{p,r}$.\ Then, monotonicity and convexity of $\sigma$ imply:
\[
\frac{1}{p} \le \fint_{B_r(x)}\sigma\left(\left|Du(y) - \fint_{B_r(x)}Du(z) dz\right|\right)dy \le \fint_{B_r(x)}\fint_{B_r(x)}\sigma\left(\left|Du(y) - Du(z)\right|\right)dydz.
\]
Hence there exists $c = c(n,p)>0$ such that for all $a \in B_{r}(x)$:
\[
c(n,p) \le \fint_{B_{2r}(a)}\fint_{B_{2r}(a)}\sigma\left(\left|Du(y) - Du(z)\right|\right)dydz.
\]
In the next lines, $c(n,p)$ may decrease, but for the sake of brevity we will not denote it differently.\ Integrating over every such $a$ and changing variables, we see that:
\begin{equation}\label{eq:r}
r^nc(n,p) \le \int_{B_r(x)}\fint_{B_{2r}}\fint_{B_{2r}}\sigma\left(\left|Du(y + a) - Du(z + a)\right|\right)dydzda
\end{equation}
As $r$ is fixed and $E_{p,r} \subset B_1$, we can cover $E_{p,r}$ with finitely many balls $\overline{B_{r}(x_i)}$ centered at $\{x_i\}_{i=1}^N \subset E_{p,r}$.\ From Vitali's Covering Theorem, \cite[Theorem 1.24]{Evans2015}, we find a subset of centers $x_{i_1},\dots, x_{i_M}$ such that
\begin{equation}\label{e:sp}
E_{p,r} \subset \bigcup_{i = 1}^N \overline{B_{r}(x_i)} \subset \bigcup_{j = 1}^M \overline{B_{5r}(x_{i_j})}\quad \text{ and } \quad B_{r}(x_{i_j}) \cap B_{r}(x_{i_k}) = \emptyset \quad \text{ if }j \neq k.
\end{equation}
Let $E \doteq \bigcup_{j = 1}^M B_r(x_{i_j})$.\ Summing \eqref{eq:r} over this subset of centers, we obtain
\begin{equation}\label{e:Mrc}
Mr^nc(n,p) \le \int_{E}\fint_{B_{2r}}\fint_{B_{2r}}\sigma\left(\left|Du(y + a) - Du(z + a)\right|\right)dydzda
\end{equation}
For any $y,z \in B_{2r}$, and since $10r \le 1$, we get
\[
\int_{E}\sigma\left(\left|Du(y + a) - Du(z + a)\right|\right)da \le \int_{B_{\frac{11}{10}}}\sigma\left(\left|Du(y + a) - Du(z + a)\right|\right)da.
\]
Changing variables and renaming $z- y \doteq h$, we finally obtain
\[
\int_{B_{\frac{11}{10}}}\sigma\left(\left|Du(y + a) - Du(z + a)\right|\right)da \le \int_{B_{\frac{3}{2}}}\sigma\left(\left|Du(a+h) - Du(a)\right|\right)da \overset{\eqref{e:Dh}}{\le} C|h|D(|h|) \le CrD(4r),
\]
for a dimensional constant $C$.\ Combining this inequality with \eqref{e:Mrc}, we infer that
\begin{equation}\label{e:rnp}
	Mr^{n - 1} \le C(n,p)D(4r).
\end{equation}
Therefore, from \eqref{e:sp} and \eqref{e:rnp} we infer
\[
\mathcal{H}^{n-1}_{5r}(E_{p,r}) \le \sum_{j=1}^M \diam(B_{5r}(x_{i_j}))^{n-1} = (10)^{n-1}Mr^{n-1} \overset{\eqref{e:rnp}}{\le} C(n,p)D(4r),
\]
for a possibly larger constant $C(n,p)$.\ This and \eqref{eq:Dh} show \eqref{eq:ficlar} and conclude the proof.
\end{proof}

\section{Proof of Theorem \ref{t:Su}}\label{sec:proofthm1}

Before showing the Theorem, let us introduce some useful notation.\ Given a map $w \in \Lip(B_1,\R^2)$ and any point $x_0 \in B_1$, we set:
\begin{equation}\label{e:resc}
	w_{r,x_0}(h) \doteq \frac{w(x_0 + rh)- w(x_0)}{r}.
\end{equation}
As $w$ is Lipschitz, this sequence of rescalings is precompact in $C^0(B)$, for any ball $B \subset \R^2$. Therefore, we can introduce the set $\mathcal{B}(w)(x_0)$ of blowups of $w$ at $x_0$ to be the collection of maps obtained as locally uniform limits of any subsequence extracted from $w_{r,x_0}$.\ Observe that, by Rademacher's theorem, $w$ is differentiable at a.e. $x_0 \in B_1$ and thus $\mathcal{B}(w)(x_0) = \{L_{Dw(x_0)}\}$ for any such $x_0$.\ Given now a solution $w \in \Lip(B_1,\R^2), B_1\subset \R^2$ to \eqref{e:di} with $K$ enjoying property \eqref{e:diffinc}, we further observe that:
	\begin{equation}\label{e:sconv}
		\text{ any map $w_\infty \in \mathcal{B}(w)(x_0)$ still solves \eqref{e:di}},
	\end{equation}
	due to the strong $W^{1,1}_{\loc}$ convergence of the gradients of rescalings provided by Theorem~\ref{t:comp}.\ Through blowups, we define the regular and the singular sets of $w$:
\begin{Def}\label{def:reg}
For $w \in \Lip(\Omega,\R^2)$, $\Omega \subset \R^2$ open, we let
		\[
		\Reg(w) \doteq \{x \in \Omega: \mathcal{B}(w)(x) \text{ contains a map which is differentiable at $0$}\}\; \text{ and } \; \Sing(w) \doteq \Omega\setminus \Reg(w).
		\]
	\end{Def}  
	
Combining now Theorems \ref{t:noninc}-\ref{p:sing}, we get the following precise version of Theorem \ref{t:Su}, restated in terms of differential inclusion thanks to Propositions \ref{p:11}-\ref{p:12}.

\begin{thm}\label{t:summa}
	Let $K \subset \R^{2\times 2}$ be compact and fulfill \eqref{e:diffinc} and let $w \in \Lip(B_1,\R^2)$, $B_1\subset \R^2$, solve \eqref{e:di}.\ Then, $w$ is differentiable over $\Reg(w)$, $Dw|_{\Reg(w)}$ is continuous, and $\mathcal{H}^1(\Sing(w)) = 0$.
\end{thm}
\begin{proof}
Let $x_0 \in \Reg(w)$.\ By definition, there exists $v \in \mathcal{B}(w)(x_0)$ which is differentiable at $0$.\ By \eqref{e:sconv}, $Dv \in K$ a.e. in $\R^2$.\ Since $v$ is differentiable at $0$, Theorem \ref{t:noninc} applied to the family $v_{r,0}$, compare \eqref{e:resc}, implies that $\diam([Dv](0)) = 0$.\ Now, as $v \in \mathcal{B}(w)(x_0)$, then $v$ is the local uniform limit in $\R^2$ of $(w_{r_j,x_0})_{j \in \N}$, and we can employ the same reasoning to deduce from Theorem \ref{t:noninc} that $\diam([Dw](x_0)) = 0$.\ Therefore, it consists of a single matrix, and differentiability and continuity at that point follow directly from the definition \eqref{e:point}. 
 To conclude the proof, 
it suffices to show that $\Sing(w) \subset S$, where $S$ is the
$\mathcal H^1$-negligible set provided by Theorem~\ref{p:sing}.
Here, Theorem~\ref{p:sing} is applied
to the first component $u$  of $w=(u,v)$,
which solves \eqref{e:G} for the strictly monotone field $G$ 
associated to $K$ by Proposition~\ref{p:11}.\ Equivalently, we show that $S^c \subset \Reg(w)$.\ 
Let then $x_0 \in S^c$.\ This means that there exists a sequence $r_n \to 0$ such that
\begin{equation}\label{e:Brn}
	\lim_{n\to +\infty} \fint_{B_{r_n}(x_0)}\sigma\left(\left|Du(y) - \fint_{B_{r_n}(x_0)} Du(z)dz\right|\right)dy = 0.
\end{equation}
We may assume, up to a non-relabeled subsequence, that $\fint_{B_{r_n}(x_0)} Du \to a \in \R^2.$ Up to considering another non-relabeled subsequence, we can assume that $w_{r_n,x_0}$, recall \eqref{e:resc}, converges locally uniformly and, thanks to Theorem~\ref{t:comp}, strongly in $W^{1,1}_{\loc}$ to $w_\infty = (u_\infty, v_\infty)$.\ Now the strong convergence and \eqref{e:Brn} imply that $Du_\infty \equiv a$ on $B_1$.\ From \eqref{e:diffinc}-\eqref{e:sigmaprop}-\eqref{e:sconv}, we  infer that $Dw_\infty \equiv A$ in $B_1$, where $A \in K$ is the only matrix in $K$ whose first row is $a$.\ Hence $w_\infty$ is differentiable at $0$, and $x_0 \in \Reg(w)$.
\end{proof}

\section{Monge-Ampère measure associated to solutions of elliptic PDEs}\label{sec:2bis}

For $u,v \in W^{1,2}(B_1)$, we define the distribution
\begin{equation}\label{e:vwj}
\mathcal{D}(u,v) = -\frac{1}{4}\dv(\dv (JDu\otimes JDv + JDv\otimes JDu)),
\end{equation}
where $J$ is defined in \eqref{e:J}.\ Recalling \eqref{eq:A}, if $u,v \in W^{2,2}$, a direct computation shows that $\mathcal{D}(u,v) = \frac{1}{2}\langle D^2u, \cof D^2v\rangle$, so that $\mathcal{D}(u,u) = \det(D^2u)$.\ In the literature, $\mathcal{D}(u,u)$ is usually called the \emph{very weak} Hessian Jacobian of $u$.\ The main result of this section is the following:

\begin{thm}\label{thm:u}
	Let $K \subset \R^{2\times 2}$ fulfill \eqref{e:diffinc}, $\Omega \subset \R^2$ be open, and let $w = (u,v) \in \Lip(\Omega,\R^2)$ be a solution to \eqref{e:di}.\ Define the symmetric matrix of distributions
	\begin{equation}\label{e:dw}
		\mathcal{D}(w)\doteq \left(\begin{array}{cc} \mathcal{D}(u,u) & \mathcal{D}(u,v)\\ \mathcal{D}(u,v) & \mathcal{D}(u,v)\end{array}\right).
	\end{equation}
	Then, for any $\varphi \in C^\infty_c(\Omega)$ with $\varphi \ge 0$ everywhere,
	\[
	\mathcal{D}(w)(\varphi) \le 0 \text{ in the sense of quadratic forms}.
	\]
	Thus, $\mathcal{D}(w)$ is a locally finite measure on $\Omega$ with values in the set of nonpositive semidefinite symmetric matrices, $\Sym^-(2)$.

\end{thm}

\begin{proof}
	We can assume $\Omega = B_1$.\ We only need to show that the distribution $\mathcal{D}(u,u)$ is a non-positive measure.\ Let us show how to conclude the proof assuming this claim. Consider $\tilde w = wA$, for $A \in \R^{2\times 2}$ with $\det(A) > 0$. Then, $D\tilde w \in \tilde K \doteq A^TK$, which fulfills \eqref{e:diffinc}, and by the claim applied to $\tilde w$ we deduce that:
	\[
	\mathcal{D}(au+bv,au +bv) = a^2\mathcal{D}(u,u) + 2ab\mathcal{D}(u,v) + b^2\mathcal{D}(v,v) \le 0, \quad \forall a,b \in \R,
	\]
	which would conclude the proof.\ Let us show the claim.\ By Proposition \ref{p:11} we have that $u$ is a Lipschitz solution to \eqref{e:G} with $G\in C^0(\R^2,\R^2)$ satisfying \eqref{e:monoG}. 
	By \cite[Lemma~A.4]{Lacombe2024}, $u$ is a strong $H^1$ limit of smooth functions $u_\e$ solving 
	\begin{equation}\label{e:ueps}
	\dv(G^\eps(Du^\eps)) = \langle DG^\eps(Du^\eps), D^2u^\eps\rangle = 0\,,
	\end{equation}
where $G^\e$ are smooth fields satisfying \eqref{e:monoG} with $\sigma(t)=c_\eps t^2$ for some $c_\eps >0$, hence $DG^\e +(DG^\e)^T \geq 2 c_\e$.
This and \eqref{e:ueps} imply that $\mathcal{D} (u^\eps,u^\eps) =\det(D^2u^\eps)\le 0$ in $\mathcal{D}'(B_1)$.\ By the strong $H^1$ convergence provided by \cite[Lemma~A.4]{Lacombe2024}, we deduce the same for $\mathcal{D}(u,u)$.
\end{proof}

\begin{cor}\label{cor:2}
Let $K \subset \R^{2\times 2}$ fulfill \eqref{e:diffinc}, and let $w = (u,v) \in \Lip(B_1,\R^2)$ be a solution to \eqref{e:di}.\ Consider the matrix-valued measure $\mathcal{D}(w)$ defined as in \eqref{e:dw}.\ Factorize it as $\mathcal{D}(w) = P \mu$, where $\mu$ is a finite, positive measure on $B_1$, and $P \in L^\infty(B_1,\Sym^-(2);\mu)$.\ Then, for any $w_\infty \in \mathcal{B}(w)(x_0)$, recall \eqref{e:sconv},
\begin{equation}\label{e:Pmu}
\mathcal{D}(w_\infty) = P(x_0)\mu(\{x_0\})\delta_0,
\end{equation}
i.e. $\mathcal{D}(w_\infty) = P(x_0)\mu(\{x_0\})$ if $\mu(\{x_0\}) \neq 0$ and $\mathcal{D}(w_\infty) = 0$ if $\mu(\{x_0\}) = 0$.
\end{cor}

\begin{proof}
Recalling \eqref{e:resc}, by definition $w_\infty = \lim_n w_{x_0,r_n}$, for a sequence $r_n \to 0$, where the convergence is locally uniform and in $W^{1,p}_{\loc}$ for all $p \in [1,\infty)$ in $\R^2$, see Theorem~\ref{t:comp}.\ Take any $\varphi \in C^\infty_c(\R^2)$ and set $\varphi_n(x) \doteq \varphi\left(\frac{x-x_0}{r_n}\right).$ Notice that $\varphi_n$ converges pointwise to $0$ in $\{x_0\}^c$ and to $\varphi(0)$ at $x_0$. On one hand: 
\begin{equation}\label{e:fico}
\lim_n\mathcal{D}(w)(\varphi_n) = \mathcal{D}(w_\infty)(\varphi),
\end{equation}
because of $\mathcal{D}(w)(\varphi_n) = \mathcal{D}(w_{x_0,r_n})(\varphi)$ and the strong $W^{1,2}$ convergence of $w_{x_0,r_n}$ to $w_\infty$.\ On the other hand:
\[
\mathcal{D}(w)(\varphi_n) = \int_{B_1}\varphi_n(x) P(x) d\mu(x),
\]
so that dominated convergence and \eqref{e:fico} conclude the proof.
\end{proof}

Let us end this section with a few general comments.\ Firstly, similar properties hold for infinity harmonic functions \cite[Theorem~1.5]{Koch2019}.\ Secondly, these results suggest the following conjectural analysis of blowups: one could hope to deduce from Corollary \ref{cor:2} that $w_\infty \in \mathcal{B}(w)(x_0)$ is linear at points where $\mu(\{x_0\}) = 0$ and that $w_\infty$ is one-homogeneous at points where $\mu(\{x_0\}) \neq 0$.\ This would tell us that $\Sing(w)$ is actually countable, since a finite measure can have at most countably many atoms.\ The reason for this hope is that the equation $\det(D^2u) = 0$ in an open set $\Omega$ implies a very rigid structure of $Du$: essentially, that $Du$ is constant along segments connecting the boundary.\ This, in turn, yields that $Du$ is constant if $u$ is defined in $\R^2$.\ Such rigidity property has been shown assuming $u \in C^2$ \cite{Hartman1959}, $u \in W^{2,\infty}$ \cite[\textsection 2.6]{Kirchheim2003}, $u \in W^{2,2}$ \cite{Pakzad2004} and $u$ in the class \emph{MA} \cite{Jerrard2010}.\ 
In all of these cases, $Du$ is (at least weakly) differentiable.\ Here, however, the only way we can interpret the distribution $\det(D^2u)$ is in the \emph{very weak sense}, see the distribution $\mathcal{D}(u,u)$ in \eqref{e:vwj}.\ The equation $\mathcal{D}(u,u) = 0$ is extremely flexible, even if $u \in C^{1,\alpha}$ for $\alpha > 0$, see \cite{Lewicka2017,Cao2019,Cao2025}.\ In a sense, a function $u$ which satisfies  $\mathcal{D}(u,u) = 0$
and also \eqref{e:G} for a strictly monotone $G$, lies in a realm between the known rigidity and flexibility statements.\ The fact that some rigidity can be expected can be seen in Proposition \ref{p:blowups_hypK*}, where Corollary \ref{cor:2} will be instrumental to exclude some types of blowups.\

\section{Regularity of inclusions into curves}\label{sec:curves}
In this section we study inclusions into elliptic curves.\ Namely, we assume that $K=\Gamma\subset\R^{2\times 2}$ satisfies \eqref{e:diffinc} and is a $C^1$ curve, that is, a connected, compact $C^1$ submanifold of $\R^{2\times 2}$ of dimension one, which may have a nonempty boundary.\ Hence $\Gamma=\gamma(I)$ for some $I=[a,b]$ or $\R/L\Z$ and $\gamma\in C^1(I,\R^{2\times 2})$ 
is a homeomorphism onto $\Gamma$ with $|\gamma'|>0$ on $I$.\ As some results depend on the global geometry of the domain, it is convenient to consider a general open, connected, simply connected set $\Omega$ in \eqref{e:GAMMA}.\ As previously mentioned, the next Theorem generalizes the main result of \cite{LLP25}, where the same conclusion was obtained under the assumption that \eqref{e:diffinc} is satisfied with $\sigma(t)=ct^4$ for some $c>0$.
		
	\begin{thm}\label{tint:3}
		Let $\Omega \subset \R^2$ be open.\ If $w \in \Lip(\Omega,\R^2)$ satisfies 
		\begin{equation}\label{e:GAMMA}
			Dw \in \Gamma \text{ a.e. in $\Omega$},
		\end{equation}
		then $\Sing(w)$ is locally finite.\ If $\Omega$ is convex, then $\Sing(w)$ contains at most two points.\ 
		Moreover, $\Sing(w)$ is empty 
		if one of the following sufficient conditions is satisfied.
		\begin{itemize}
			\item $\Gamma$ is simply connected, that is, $I=[a,b]$;
			\item $\Gamma$ is not fully degenerate, that is, it admits a tangent line generated by a rank-2 matrix;
			\item there exists $a\in\mathbb S^1$ such that the projection  $a\Gamma\subset\R^2$ is the boundary of a strictly convex open set. 
		\end{itemize}
	
	\end{thm}
	
	In fact, under the last assumption, the conclusion holds even if $\Gamma$ is not assumed $C^1$, see Corollary~\ref{c:visc}.\ Aside from its independent interest, this result will be crucial for the proof of Theorem \ref{tint:4}.\ As the proof is quite long, it is useful to split it into steps:

\begin{itemize}
	\item First, in \textsection \ref{sub:step1}, we show that, for any $a \in\R^2$, $ \varphi \doteq (a,w)$ is a viscosity solution in the sense of \cite{Crandall1983} to two Hamilton-Jacobi equations simultaneously: $\pm f(D\varphi) = 0$, 
where $f$ depends only on $\Gamma$ and $a$; 
	\item next, in \textsection \ref{sub:step2} we will reduce the proof to the case where $\rank(\gamma'(t)) = 1$, for all $t \in I$;
	\item \textsection \ref{sub:step3} considers a special case, the one where $\Gamma$ is \emph{graphical}, see the special form \eqref{eq:gamma_graph}.\ The main result of that subsection is that, for such curves, solutions to \eqref{e:GAMMA} are everywhere $C^1$; 
	\item In \textsection \ref{sub:step4}, we deduce as a consequence the partial regularity result for general curves $\Gamma$;
	\item The fifth and sixth step, contained in \textsection \ref{sub:step5}-\ref{sub:step6}, are devoted to show that $w$ is a  \emph{zero-entropy solution}  to $Dw \in \Gamma$ and to deduce from this the structure of the singular set, respectively.
\end{itemize}
Theorem~\ref{tint:3} is then a consequence of Corollary~\ref{c:visc}, Lemma~\ref{l:reduc_degen}, Propositions~\ref{p:reg_open_curve}-\ref{p:reg_closed_curve} and Lemma~\ref{l:struct_isolated_sing}.

\subsection{Step 1: Viscosity properties}\label{sub:step1}

We wish to establish the following viscosity-type property.

\begin{prop}\label{p:visc}
Assume $K\subset\R^{2\times 2}$ is compact and  satisfies \eqref{e:diffinc}.
Let $w\in \Lip(\Omega,\R^2)$ solve $D w\in K$ a.e. and define, for any $a \in \R^2$, $\varphi \doteq w\cdot a$.\ Then, $\varphi$ is a strong viscosity solution of the differential inclusion $D\varphi\in a K$, i.e. if $\zeta\in C^1( \Omega)$ is such that $\varphi-\zeta$ has an extremum at $x_0\in\Omega$, then $D\zeta(x_0)\in a K \subset \R^2$.
\end{prop}

Let us first recall the following properties of probability measures on $K$.

\begin{lem}
\label{l:avg}
If $K\subset\R^{2\times 2}$ satisfies \eqref{e:diffinc}, then for any probability measure $\mu$ supported in $K$  we have
\begin{align}\label{eq:Aomega}
\mathcal A(\mu) \doteq \int \det (X)\, d\mu(X) -\det\Big( \int X d\mu(X)  \Big) \geq \omega\Big( \int \Big|X-\int Y d\mu(Y) \Big|^2\, d\mu(X) \Big)\,,
\end{align}
for some nondecreasing function $\omega\colon [0,\infty)\to [0,\infty)$ 
 such that $\omega(0)=0$ and $\omega(t)>0$ for all $t>0$.
\end{lem}

\begin{proof}
This proof is essentially contained in \cite[Lemma~3]{sverak93}.\ Indeed we have:
\begin{align*}
\int\!\!\! \int \det(X-Y)\,d\mu(X)d\mu(Y)
&
=\int\!\!\!\int 
\Big(
\det(X) +\det(Y)-\langle X,\cof(Y)\rangle
\Big)\, 
d\mu(X) d\mu(Y)
\\
&
=2\int \det(X)\, d\mu(X) - 
\left\langle\int X\, d\mu(X),\cof\left(\int Y\,d\mu(Y)\right)\right\rangle = 2\mathcal{A}(\mu).
\end{align*}
Thus $\mathcal A(\mu) \ge 0$, with equality if and only if $\mu\otimes\mu$ is supported on the diagonal, which is equivalent to $\mu$ being a Dirac mass.
Then one can take 
$\omega(t)$ to be the minimum of the
 weak-$*$ continuous function $\mathcal A$ over the weak-$*$ compact set of probability measures $\mu$ on $K$ such that 
 \[
 \int |X- \int Y d\mu(Y)|^2\, d\mu(X) \geq t.\tag*{\qedhere} 
 \]
\end{proof}

\begin{proof}[Proof of Proposition~\ref{p:visc}]
If $a = 0$ there is nothing to show. If $a \neq 0$, then we can write $a = e_1 Q$, for some $Q\in \R^{2\times 2}$ with $\det(Q)>0$.\ The map $\tilde w =wQ $ satisfies $D\tilde w\in \tilde K =Q^TK$, and $\tilde K$ still satisfies \eqref{e:diffinc}, so it suffices to prove Proposition~\ref{p:visc} for $a = e_1$.\ Letting, as usual $w = (u,v)$, we then wish to show the property for $\varphi = u$. Let us notice that, applying the previous observation with $ a = \pm e_1$, it suffices to consider the case where the extremum of $u-\zeta$ is a minimum.\ As a further simplification we can, without loss of generality, assume that such minimum point is at $x_0 = 0$, that $u(0) = \zeta(0)$, and that the minimum is strict. Indeed, if the latter is not satisfied, then we can simply replace $\zeta$ by $\zeta_\e(x)=\zeta(x)+\e|x|^2$ for $0<\e\ll 1$.

\medskip

Thanks to the previous reductions, we now fix $\zeta\in C^1(B_1)$ such that $u-\zeta$ has a strict minimum at $0$, and we let $K_1 = \pi_1(K)\subset\R^2$, as in Proposition \ref{p:11}.\ To show $D\zeta(x_0)\in K_1$, we follow the strategy of \cite[\textsection{4.2}]{DLOW04}. 
For small $\delta>0$, we consider the open set $\Omega_\delta\subset\subset \Omega$ given by the connected component of $\lbrace  u-\zeta < \delta\rbrace$ which contains $0$.
Denote by $\nu_\delta$ and $\mu_\delta$ the following probability measures: 
\begin{align*}
\nu_\delta =\frac{1}{|\Omega_\delta|}\mathbf 1_{\Omega_\delta} \, dx\, \text{ and } \mu_\delta = (D w)_\sharp  \nu_\delta\,.
\end{align*}
Notice that $\nu_\delta$ is a probability measure on $\Omega$, while $\mu_\delta$, its pushforward through $Dw$, is a probability measure on $K$.\ Since $u-\zeta-\delta=0$ on $\partial\Omega_\delta$,
the divergence theorem and $\curl( D v)=0$ imply, respectively:
\begin{equation}\label{e:bary}
\int \partial_1 (u-\zeta -\delta) \, d\nu_\delta = \int \partial_2(u-\zeta-\delta)\, d\nu_\delta =0\, \text{ and } \int \partial_2 v \partial_1(u-\zeta-\delta)\, d\nu_\delta 
=\int \partial_1 v \partial_2(u-\zeta-\delta)\, d\nu_\delta\,.
\end{equation}
From these two identities we readily infer that:
\begin{align*}
\mathcal A(\mu_\delta) \overset{\eqref{eq:Aomega}}{=} \int
 \det( D w)
\,d\nu_\delta
-\det\Big(\int  D w \,d\nu_\delta\Big)
=
\int \partial_2v
\Big(\partial_1\zeta
-\int\partial_1\zeta \,d\nu_\delta\Big)
\, -
\partial_1v
\Big(\partial_2\zeta
-\int\partial_2\zeta\,d\nu_\delta\Big)
\,d\nu_\delta\,.
\end{align*}
Since $\zeta \in C^1$ and $\mathrm{diam}(\Omega_\delta)\to 0$ as $\delta\to 0$, we deduce $\mathcal A(\mu_\delta)\to 0$ as $\delta\to 0$.\ Notice that the properties of $\omega$ in \eqref{eq:Aomega} ensure that if $\omega(t_k)\to 0$, then $t_k\to 0\,$, for any sequence $t_k \ge 0$.\ Thus, $\mathcal A(\mu_\delta)\to 0$ and \eqref{eq:Aomega} imply:
\begin{align*}
&
\int \Big| D u(x)-\int  D u(y)\, d\nu_\delta(y) \Big|^2\, d\nu_\delta(x)
\longrightarrow 0
\quad\text{as }\delta\to 0\,.
\end{align*}
Since $ D u\in K_1$ a.e. and $\int D \zeta\,d \nu_\delta =\int  D u\, d\nu_\delta$ by the first equation of \eqref{e:bary}, we infer $$\dist\left(\int  D \zeta\, d\nu_\delta,K_1\right)\to 0\,.$$
As $K_1$ is closed, $ D\zeta \in C^0(\Omega,\R^2)$, 
$\lim_{\delta \to 0}\diam(\Omega_\delta) = 0$ and $0 \in \cap_{\delta > 0}\Omega_\delta$,
we conclude that $ D\zeta(0)\in K_1$.
\end{proof}

It is important to note that the viscosity property established in Proposition~\ref{p:visc} is very strong:
for any continuous function $f\colon\R^2\to \R$ that vanishes on $aK$,
the function $\varphi=w\cdot a$ is a viscosity solution of $f(D\varphi)=0$, but also of $-f(D\varphi)=0$.
If $f$ is strictly convex, this already implies $C^1$ regularity.

\begin{cor}\label{c:visc}
Let $K\subset\R^{2\times 2}$ be a compact set fulfilling \eqref{e:diffinc}.
Assume that there exists $a\in\mathbb S^1$ such that $aK\subset\R^2$ is contained in the boundary of a strictly convex open set.
Then any $w \in \Lip(B_1,\R^2)$ solving \eqref{e:di} is $C^1$ in $B_1$.
\end{cor}
\begin{proof}
A strictly convex open subset of $\R^2$ is the sublevel set of its  gauge function with respect to any interior point, see e.g. \cite[Lemma~1.2]{brezis},
so $aK$ is contained in the zero set of a strictly convex Lipschitz function $f\colon\R^2\to\R$.
We assume without loss of generality that $a=e_1$.
According to Proposition~\ref{p:visc}, the function $\varphi=w_1$ is a viscosity solution of $f(D\varphi)=0$, 
and so is $\tilde\varphi(x)=-\varphi(-x)$.
By \cite[Theorem~5.3.7]{CS04} this implies that both $\varphi$ and $\tilde\varphi$ are locally semiconcave in the sense of \cite[Definition~2.1.1]{CS04}, hence $\varphi$ is  both locally semiconcave and locally semiconvex.
We infer that $\varphi$ is $C^1$ by \cite[Theorem~3.3.7]{CS04},
and finally that $w$ is $C^1$ since the differential inclusion $Dw\in K$ gives
 $Dw_2$ as a continuous function of $Dw_1$, see Proposition~\ref{p:11}.
\end{proof}

\subsection{Step 2: Reduction to degenerate curves}\label{sub:step2}

\begin{lem}\label{l:reduc_degen}
If $w \in \Lip(\Omega,\R^2)$ satisfies \eqref{e:GAMMA}, then either $Dw$ is constant in $B_1$ or $Dw\in \Gamma_*$ a.e., where $\Gamma_*=\gamma(I_*)$ for some $I_*=[a_*,b_*]$ or $\R/L\Z$ and $\det(\gamma')=0$ on $I_*$.
\end{lem}
\begin{proof}[Proof of Lemma~\ref{l:reduc_degen}]
Thanks to the unique continuation result of \cite{DPGT24}, see \cite[Proposition~3.1]{LLP25} and Remark~\ref{r:rigid_est}, 
$Dw$ is either constant in $\Omega$ or takes values into the degenerate part $\Gamma_{d}=\gamma(\lbrace \det(\gamma')=0\rbrace)$.\ We can assume that we are in the latter case.\ According to Theorem \ref{theo:connected}, $Dw$ takes values at almost every point in one single connected component $\mathcal C$ of $\Gamma_d$.\ The required interval is then $I_* \doteq \gamma^{-1}(\mathcal C)$.\end{proof}

\begin{rem}\label{r:rigid_est}
Note that \cite[Proposition~3.1]{LLP25} is stated for $C^2$ curves. The results of \cite{LLP25} are restricted to $C^2$ curves for other reasons, but \cite[Proposition~3.1]{LLP25} relies on \cite{DPGT24}, which does not require any smoothness, and on \cite{LLP24}, where the proof is written for smooth curves, but works for $C^1$ curves modulo minor adaptations.
\end{rem}

Thanks to Lemma~\ref{l:reduc_degen} we assume from now on, in addition to \eqref{e:diffinc}, that $K=\Gamma=\gamma(I)$ for some $I=[a,b]$ or $\R/L\Z$ and $\gamma\in C^1(I,\R^{2\times 2})$ is a homeomorphism onto $\Gamma$ with 
\begin{equation}\label{e:gamma0}
	\text{$|\gamma'|>0$ and $\det(\gamma')=0$ on $I$.}
\end{equation}

\subsection{Step 3: Degenerate graphical curves}\label{sub:step3}

In this subsection we assume that $I=[a,b]$ and that $\Gamma$ is a graph over one of its four components.\ Without loss of generality, we assume that:
\begin{align}\label{eq:gamma_graph}
\gamma(t)
=
\left(
\begin{array}{cc}
-f(t) & t
\\
-q(t) & \eta(t)
\end{array}
\right)
\quad\text{for }t\in I\,,
\end{align}
for some $f,\eta,q\in C^1(I)$ satisfying, thanks to Lemma \ref{l:reduc_degen}, the degeneracy condition $\det(\gamma')=q'-\eta'f'=0$.\ This, together with the ellipticity condition \eqref{e:diffinc}, implies that $f$ cannot be affine on any open interval.

\begin{prop}\label{p:reg_curv_graph}
If $w\colon \Omega\to\R^2$ satisfies $Dw\in\Gamma$ a.e. in $\Omega$, then $w$ is $C^1$, and $Dw$ is constant along characteristic lines directed by $(1,f'(\partial_2w_1))$, if $w = (w_1,w_2)$.
\end{prop}
\begin{proof}
%
Let $h=w_1$.\ By Proposition~\ref{p:visc} $h$ is a viscosity solution in $\Omega$ of the Hamilton-Jacobi equation $\partial_1 h +f(\partial_2 h)=0$.\ This implies that $u=\partial_2 h$ is an entropy solution of the scalar conservation law
\begin{align}\label{eq:scl}
\partial_1 u + \partial_2 f(u)=0\,,
\end{align}
see e.g.  Lemma~\ref{l:visc-ent} below.
In other words, all entropy productions of $u$ are nonpositive distributions:
\begin{align*}
\partial_1 A(u)+\partial_2 B(u)\leq 0\,, \text{ for all $A,B\in C^1(\R)$ such that $B'=f'A'$.}
\end{align*}
Since $\tilde w(x)=-w(-x)$ satisfies the same differential inclusion, the function $\tilde u(x)=u(-x)$ also has this property.\ As a consequence, all entropy productions of $u$ are both nonpositive and nonnegative:
\begin{equation}\label{e:entrograph}
\partial_1 A(u)+\partial_2 B(u)= 0\quad\text{in }\mathcal D'(\Omega)\,, \text{ for all $A,B\in C^1(\R)$ such that $B'=f'A'$.}
\end{equation}
Fixing any $v\in\R$ and continuous $\rho$ with support in $(0,1)$ and $\int_0^1\rho=1$, we use the previous equality with
\begin{align*}
	A_k(t)=\int_{-\infty}^{t-v} k\rho(ks)\,ds,
	\quad
	B_k(t)=\int_{-\infty}^t f'(s)A_k'(s)\, ds.
\end{align*}
It is not hard to check that $A_k(t)\to\mathbf \mathbf 1_{\{t>v\}}$ and $B_k(t)\to f'(v)\mathbf 1_{\{t>v\}}$ for all $t\in\R$ as $k \to \infty$.\ By dominated convergence, we deduce
that the indicator function $\chi_v(x)\doteq \mathbf 1_{\{u(x)>v\}}$ solves the free transport equation
\begin{align}\label{eq:kin}
\partial_1 \chi_v +f'(v)\partial_2\chi_v =0\quad\text{in }\mathcal D'(\Omega)\,,
\end{align}
for all $v\in\R$.
This is  the kinetic formulation \cite{LPT94}
associated with the scalar conservation law \eqref{eq:scl}, in the case of zero entropy production.\ The fact that $f$ cannot be affine on any open interval ensures that for any $v_1<v_2\in  [a,b]$ one can find $v_1<\tilde v_1 <\tilde v_2<v_2$ such that $(1,f'(\tilde v_1))$ and $(1,f'(\tilde v_2))$ are linearly independent in $\R^2$.\ This property, combined with the free transport equation, implies that $u$ is continuous, see e.g. \cite[Proposition~6]{DLOW03}.
We sketch here the argument for the readers' convenience: 
to prove e.g. upper semicontinuity
it suffices to show, for all $v_1<v_2$, the existence of $0<\delta<1$ such that
\begin{align*}
u(x_0)\leq v_1 \quad\Longrightarrow
\quad u\leq v_2\text{ a.e. in }B_{\delta r}(x_0)\,,
\end{align*}
for any Lebesgue point $x_0\in\Omega$ with $B_{r}(x_0)\subset\Omega$.
(Upper semicontinuity follows thanks to the fact that $\delta$ is independent of the Lebesgue point $x_0$, see \cite{DLOW03} for details, and lower semicontinuity is proved in the same way.)
To prove this implication, note that
$x_0$ is a Lebesgue point of $\chi_{\tilde v_1}$ with value $\chi_{\tilde v_1}(x_0)=0$.
Using  \eqref{eq:kin} we deduce that all points in the line interval $\tilde I = [x_j+\R (1,f'(\tilde v_1))]\cap B_{r}(x_0)$ are Lebesgue points of $\chi_{\tilde v_1}$, with value $0$.
As $\tilde v_2 > \tilde v_1$, they are also Lebesgue points of $\chi_{\tilde v_2}$ with value $0$, and using again \eqref{eq:kin} we see that all points in the set $[\tilde I +\R(1,f'(\tilde v_2))]\cap B_r(x_0)$ are Lebesgue points of $\chi_{\tilde v_2}$ with value $0$. 
Since $(1,f'(\tilde v_1))$ and $(1,f'(\tilde v_2))$ are linearly independent, 
this set contains $B_{\delta r}(x_0)$ for some $\delta\in (0,1)$ depending on $\tilde v_1,\tilde v_2$ (hence only on $v_1,v_2$), and we conclude that $u\leq v_2$ a.e. in $B_{\delta r}(x_0)$.

 The fact that $u$ is constant along the characteristic lines directed by $(1,f'(u))$ is also a consequence of the free transport equation.\ In fact, any continuous weak solution of \eqref{eq:scl} has this property, see \cite{Dafermos2006}.\ We conclude that $Dw=\gamma(u)$ is continuous and constant along characteristics, as wanted.
\end{proof}
\begin{lem}\label{l:visc-ent}
If $f\in C^1(\R)$, $\Omega\subset\R^2$ is open and $h \in \Lip(\Omega)$ is a viscosity solution of 
\begin{align}\label{eq:HJ_time}
\partial_t h +f(\partial_x h)=0\,,
\end{align}
then $u=\partial_x h$ is an entropy solution of $\partial_t u + \partial_x [f(u)]=0\,.$
%
\end{lem}
\begin{proof}
Let $L=\|\partial_x h \|_{L^\infty(\Omega)}$ and $C=\max_{[-L,L]} |f'|\,.$ 
To show that $u$ is an entropy solution it suffices to do so in a neighborhood of any $(t_0,x_0)\in\Omega$.\ Pick $T>0$ such that $[t_0-T/2,t_0+T/2]\times [x_0-CT,x_0+CT]\subset\Omega$.\ Translating the coordinates, we assume $t_0-T/2=0$ and $x_0=0$.\ We choose a bounded, compactly supported, $L$-Lipschitz function $\tilde h_0\colon\R\to\R$ such that $\tilde h_0(x)=h(0,x)$ for $x\in [-CT,CT]$, and we consider the unique viscosity solution $\tilde h\colon [0,T]\times \R\to\R$ of \eqref{eq:HJ_time} in $[0,T]\times\R$, see for instance \cite[Theorem VI.2]{Crandall1983}.\ Thanks to the local comparison principle 
\cite[Theorem V.3]{Crandall1983}
it coincides with $h$ in the cone 
$
\mathcal C =\lbrace 0\leq t\leq T,\, |x|\leq C(T-t)\rbrace\,.
$
Moreover $\tilde h$ can be obtained as $\tilde h=\lim \tilde h_\e$, in the sense of distributions, of the vanishing viscosity approximation $\tilde h_\e$ solving
\begin{align}\label{eq:HJ_visc}
\partial_t\tilde h_\e +f(\partial_x \tilde h_\e)  =\e \partial_{xx}^2 \tilde h_\e\quad\text{in }[0,T]\times \R\,,\quad \text{ and } \quad \tilde h_\e(0,\cdot )  =\tilde h_0\quad\text{in }\R\,.
\end{align}
Thus the function $\tilde u_\e =\partial_x \tilde h_\e$ solves
\begin{align}\label{eq:scl_visc}
\partial_t\tilde u_\e +\partial_x [f(\tilde u_\e) ]  =\e \partial_{xx} \tilde u_\e\quad\text{in }[0,T]\times \R\,, \quad \text{ and } \quad \tilde u_\e(0,\cdot )  =\tilde u_0\quad\text{in }\R\,,
\end{align}
with $\tilde u_0=\partial_x \tilde h_0$.
Thanks to \cite[\S~6.3]{dafermos16}, we have $\tilde u_\e\to \tilde u$ as distributions, 
where $\tilde u$ is the unique entropy solution such that $\tilde u(0,\cdot)=\tilde u_0$.
Moreover we also have $\tilde u_\e\to \partial_x\tilde h$ as distributions, and since $\tilde h=h$ in $\mathcal C$ we deduce that $\tilde u =u$ in $\mathcal C$, so $u$ is an entropy solution in $\mathcal C$, and therefore in a neighborhood of $(t_0,x_0)$.
\end{proof}
\begin{rem}
	Lemma \ref{l:visc-ent} is surely well-known to experts, and is in fact an "if and only if" statement.\ For brevity, we only recalled the argument to show the implication we used in the proof above.
\end{rem}

\subsection{Step 4: Partial regularity}\label{sub:step4}

We infer a regularity result under small pointwise oscillation.\ Recall that we work, without loss of generality, with a curve $\Gamma$ satisfying \eqref{e:gamma0}.

\begin{prop}\label{p:reg_curve_small_osc_ptwise}
Let $\Gamma$ satisfy \eqref{e:diffinc} and \eqref{e:gamma0}.\ Then, there exists $\e = \e(\Gamma)>0$ such that, if $w \in \Lip(\Omega,\R^2)$ solves \eqref{e:GAMMA} and $\diam([Dw](\Omega))\leq \e\,$,
then $w \in C^1(\Omega,\R^2)$, and is constant along characteristic lines: for any $x\notin \Sing(w)$, the matrix $Dw$ is 
constant on the connected component of $(x+\R v)\cap \Omega $ containing $x$, where $v$ is any vector $v\in \ran(\cof M)^T$ for $M\in T_{Dw(x)}\Gamma$.
\end{prop}
\begin{proof}
There exists $\e>0$ such that, for any $M\in\Gamma$, the intersection $\Gamma\cap \overline {B_\e(M)}$ is a graph over one of its 4 components. Therefore the small oscillation assumption $\diam([Dw](B_1))\leq\e$ allows to assume that $\Gamma$ is such a graph. Possibly permuting coordinates in the domain and in the target, we may moreover assume that $\Gamma$ is a graph over its $(1,2)$ component as in \eqref{eq:gamma_graph}.
We can then apply Proposition~\ref{p:reg_curv_graph} and deduce that $w = (w_1,w_2)$ is $C^1$ in $\Omega$, and 
constant along a segment directed by $(1,f'(\partial_2w_1))$ and containing $x$.
A direct calculation shows that, for $\gamma$ as in \eqref{eq:gamma_graph} with $q'=f'\eta'$, the range of $(\cof\gamma'(t))^T$ is the line spanned by $(1,f'(t))$,
so this is consistent with the characteristic lines described in the statement of Proposition~\ref{p:reg_curve_small_osc_ptwise}.
\end{proof}

Note that, since \eqref{e:gamma0} holds, the characteristic direction $v=\Psi(t)\in \ran(\cof\gamma'(t))^T$ is uniquely determined in the projective line $\mathbb {RP}^1$.
One can also check that it depends continuously on $t$.

\begin{lem}\label{l:psi}
There exists a unique continuous map $\Psi\colon I\to\mathbb{RP}^1$ such that $\ran(\cof\gamma'(t))^T=\R\Psi(t)$ for all $t\in\R$.
Moreover $\Psi$ is not constant on any open interval.
\end{lem}
\begin{proof}
The continuity of $\Psi$ is contained in \cite[Lemma~5.1]{LLP25},  we provide the proof for completeness.
Without loss of generality we assume that $\gamma$ is an arc-length parametrization, that is, $|\gamma'|=1$ on $I$.\ This, together with the degeneracy $\det(\gamma')=0$, implies that $\mathfrak a=[\gamma']_{\mathcal H}$ and $\mathfrak b=[\gamma']_{\bar{\mathcal H}}$ satisfy $|\mathfrak a|=|\mathfrak b|=1/2$, compare \eqref{alg}.
As continuous maps from $I$ to $\frac 12\mathbb S^1$, 
they can be written as $2\mathfrak a=e^{i\alpha}$, $2\mathfrak b=e^{i\beta}$, for some liftings $\alpha,\beta\colon I\to\R$.
These liftings are continuous in the case $I=[a,b]$, but may fail to be periodic in the non-simply-connected case $I=\R/L\Z$, so that they may not be continuous maps from $I$ to $\R$. However, in this case they can still be identified with continuous functions on $\R$ which satisfy $\alpha(t+L)=\alpha(t)+2k\pi$, $\beta(t+L)=\beta(t)+2\ell\pi$ for some $k,\ell\in\Z$.
 Direct calculation then shows that
 \begin{align*}
 \cof\gamma'(t)=ie^{i\frac{\alpha+\beta}{2}}\otimes ie^{i\frac{\beta-\alpha}{2}}\,,
 \end{align*}
 so $\ran(\cof\gamma')^T=\R\Psi$, where $\Psi=ie^{i\frac{\beta-\alpha}{2}}$ is continuous from $I$ to $\mathbb S^1$ in the case $I=[a,b]$, and identified with a continuous map from $\R$ to $\mathbb S^1$ in the case $I=\R/L\Z$.\ In this latter case it satisfies $\Psi(t+L)=e^{i(\ell-k)\pi}\Psi(t)$, so $\Psi$ is continuous when seen as a map into $\mathbb{RP}^1$.
Moreover, if $\Psi$ is constantly equal to $\Psi_0$ on an open interval $(s,t)$, then integrating $\gamma'$ over $(s,t)$ we find that
 $\ran\cof(\gamma(t)-\gamma(s))^T\subset\R\Psi_0$, in contradiction with the ellipticity assumption \eqref{e:diffinc}.
\end{proof}

Finally, Proposition~\ref{p:reg_curve_small_osc_ptwise} and our previous analysis imply the following partial regularity result.

\begin{prop}\label{p:partial_reg_curve}
Let $w$ solve \eqref{e:GAMMA}. Then $\Sing(w)$ is closed, $\mathcal H^1(\Sing(w)) = 0$ and $w\in C^1(\Omega\setminus\Sing(w))$.
\end{prop}
\begin{proof}
Thanks to Theorem \ref{t:summa}, it suffices to show that the set $\Reg(w)$ defined in Definition~\ref{def:reg} is open. Let $\e>0$ be as in Proposition~\ref{p:reg_curve_small_osc_ptwise}.\ 
Thanks to Theorem \ref{t:noninc}, if $x\in \Reg(w)$ then there exists $\delta=\delta(\e)>0$ such that $\diam([Dw](B_\delta(x)))\leq \e$, 
see the proof of Theorem \ref{t:summa} for a similar argument, so $w$ is $C^1$ in $B_\delta(x)$, hence $B_\delta(x)\subset\Reg(w)$.
\end{proof}

\subsection{Step 5: Entropy productions}\label{sub:step5}

In this section and the next one, the goal is to analyze the structure of the solution $w$ to $Dw \in \Gamma$ a.e. around singular points, thus completing the proof of Theorem~\ref{tint:3}.\ Recall that we work under assumption \eqref{e:gamma0} for $\Gamma$.\ 

\medskip

Heuristically, one can expect a rigid structure of $w$ near singularities because of the local constancy of $Dw$ along characteristic lines outside $\Sing(w)$ combined with the fact that generic lines do not intersect the $\mathcal H^1$-negligible singular set $\Sing(w)$, i.e. Propositions \ref{p:reg_curve_small_osc_ptwise}-\ref{p:partial_reg_curve}.
One difficulty is that characteristic lines are not generic since their direction depends on the value of $Dw$, so it is not at all obvious that their typical behavior is not to intersect $\Sing(w)$.
This difficulty can be overcome by taking inspiration from the kinetic formulation \eqref{eq:kin} of the scalar conservation law \eqref{eq:scl} used in the proof of Proposition~\ref{p:reg_curv_graph}.
There, the kinetic variable $v$ is decoupled from the value of the solution $u(x)$, and for generic $v$ the line $x+\R(1,f'(v))$ will not intersect a $\mathcal H^1$-negligible set.
Since the kinetic formulation \eqref{eq:kin} is related to the vanishing of entropy productions it becomes natural to introduce similar tools here.\ We refer the reader to \cite{perthame02} for a systematic treatment of the link between entropies and kinetic formulations of conservation laws.

\medskip

The notion of entropy production for the differential inclusion into $\Gamma=\gamma(I)$ has been used already in \cite{LLP25}.\ 
It stems from the
system of conservation laws $\dv\cof(Dw)=0$ 
satisfied by
the map $Dw\colon\Omega\to\Gamma$.
If $Dw$ is $C^1$, 
an application of the chain rule 
provides a whole family of conservation laws $\dv\Sigma(Dw)=0$, 
for any vector field $\Sigma\in C^1(\Gamma,\R^2)$ 
whose tangential derivative $\partial_\tau\Sigma(\gamma(t))\in\R^2$ 
is orthogonal to the kernel of $\cof\gamma'(t)$
 for all $t\in I$.
These vector fields $\Sigma$ are called \emph{entropies}
and we denote their class by
\begin{align*}
\mathcal E_\Gamma
&
=\big\lbrace
\Sigma\in C^1(\Gamma,\R^2)\colon
\partial_\tau\Sigma(A)\in (\ker(\cof M))^\perp =\ran(\cof M)^T\,,
 \;\forall A\in\Gamma,\, M\in T_A\Gamma\setminus\lbrace 0\rbrace
\big\rbrace
\\
&
=
\big\lbrace
\Sigma\in C^1(\Gamma,\R^2)\colon
(\Sigma\circ\gamma)'(t)\in (\ker(\cof \gamma'(t)))^\perp =\ran(\cof \gamma'(t))^T\,,
 \;\forall t\in I
\big\rbrace
\,.
\end{align*}
If $w$ is just a Lipschitz solution to $Dw\in\Gamma$ a.e. then there is no direct reason for the entropy productions $\dv\Sigma(Dw)$ to vanish.\ The first step is to show that this actually happens, as for smooth solutions.

\begin{prop}\label{p:ent}
If $w \in \Lip(\Omega,\R^2)$ satisfies \eqref{e:GAMMA}, then $\dv\Sigma(Dw)=0$ in $\mathcal D'(\Omega)$ for all $\Sigma\in \mathcal E_\Sigma$.
\end{prop}
\begin{proof}[Proof]
Recall from Proposition~\ref{p:partial_reg_curve} that $w$ is $C^1$ outside a closed $\mathcal H^1$-negligible set $\Sing(w)$.
We fix $\Sigma\in \mathcal E_\Gamma$ and start by showing that,
for any $x\in \Omega\setminus \Sing(w)$ there exists $r>0$ such that
$\dv\Sigma(Dw)=0$ in $\mathcal D'(B_r(x))$.
Since $Dw$ is continuous in $\Omega\setminus\Sing(w)$, we may choose $r>0$ such that $Dw(B_r(x))$ is contained in a portion $\Gamma_* = \gamma_*(I_*)\subset\Gamma$ where $\gamma^*$ can be written as in \eqref{eq:gamma_graph}, without loss of generality.\ 
With that notation, the kernel of $\cof\gamma'_*(t)$ is spanned by the vector $(-f'(t),1)$, and the condition 
$\partial_\tau\Sigma(\gamma)\perp\ker\cof(\gamma')$ therefore implies 
$\partial_\tau\Sigma_2(\gamma)=f'\partial_{\tau}\Sigma_1(\gamma)$ on $I_*$,
that is, $(\Sigma_2\circ\gamma_*)'=f'(\Sigma_1\circ\gamma_*)'$ on $I_*$.
Now recall from the proof of Proposition~\ref{p:reg_curv_graph} that $u=\partial_2 w_1$ solves \eqref{e:entrograph}, namely
\begin{align*}
\partial_1 A(u)+\partial_2 B(u)=0\quad\text{in }B_r(x)\,, \text{ for all $A,B\in C^1(\R)$ such that $B'=f'A'$.}
\end{align*}
Applying this to $A=\Sigma_1\circ\gamma$ and $B=\Sigma_2\circ\gamma$, and noting that $\Sigma(Dw)=\Sigma(\gamma(u))$, we get $\dv\Sigma(Dw)=0$ in $B_r(x)$.
This is valid for any $x\in\Omega\setminus\Sing(w)$ and some $r=r(x)>0$, hence $\dv\Sigma(Dw)=0$ in $\mathcal D'(\Omega\setminus \Sing(w))$.\ Since $\Sigma(Dw)\in L^\infty(\Omega)$ and $\mathcal H^1(\Sing(w))=0$, 
\cite[Theorem~4.1(b)]{Harvey1970} yields $\dv\Sigma(Dw)=0$ in $\mathcal D'(\Omega)$.
\end{proof}

Through this Proposition, we could in principle apply the strategy of \cite[\textsection{7}]{LLP25} to conclude the proof.\ However, in \cite{LLP25} the curve $\Gamma$ is assumed to have $C^2$ regularity, and some nontrivial adaptations are required to deal with our lower $C^1$ regularity.

\subsubsection{The case of a nonclosed curve}

In the case of a nonclosed curve, $I=[a,b]$, Proposition~\ref{p:ent} implies quite directly the continuity of  $Dw$, the argument being essentially the same as the one of Proposition \ref{p:reg_curv_graph}.

\begin{prop}\label{p:reg_open_curve}
Assume $I=[a,b]$ and $w$ solves \eqref{e:GAMMA}.\ Then $w$ is $C^1$ in $\Omega$.
\end{prop}
\begin{proof}
Since $I=[a,b]$, we can lift the map $\Psi$ from Lemma~\ref{l:psi} to
a continuous map from $I$ to $\mathbb S^1$, which we still denote by $\Psi$.
For any $\alpha\in [a,b]$, the 
map $\Sigma^\alpha\colon \Gamma\to\R^2$ given by $\Sigma^\alpha \circ\gamma(t) \doteq \Psi(\alpha) \mathbf 1_{\{t > \alpha\}}$ is a pointwise limit $\Sigma^\alpha =\lim \Sigma_j$ of entropies $\Sigma_j\in \mathcal E_\Gamma$.
One can see this for instance by setting:
\begin{align*}
\Sigma_j=\tilde\Sigma_j\circ\gamma^{-1},\quad  \text{ with }\tilde\Sigma_j(t)=\int_a^t \rho_j(s-\alpha)\Psi(s)\, ds\,,
\end{align*}
where $\rho_j(s)=j\rho(sj)$ for some continuous nonnegative function $\rho$ with support in $(0,1)$ and unit integral.
Let $\theta\colon\Omega\to I$ be such that $Dw=\gamma(\theta)$.
By Proposition~\ref{p:ent}, we have $\dv\Sigma_j(\gamma(\theta))=0$ in $\mathcal D'(\Omega)$ for all $j\geq 1$ and therefore, by dominated convergence, $\dv\Sigma^\alpha(\gamma(\theta))=0$.
This amounts to the kinetic formulation
\begin{equation}\label{e:kinfor}
\Psi(\alpha)\cdot D(\mathbf 1_{\{\theta > \alpha\}}) =0\quad\text{in }\mathcal D'(\Omega)\,, \text{ for all }
\alpha\in [a,b].
\end{equation}
Recall from Lemma~\ref{l:psi} that $\Psi$ is not constant on any open interval.
Thus, for any $\alpha_1<\alpha_2\in [a,b]$, there exist 
$\alpha_1<\tilde\alpha_1<\tilde\alpha_2<\alpha_2$ such that $\Psi(\tilde\alpha_1)$ and $\Psi(\tilde\alpha_2)$ are linearly independent in $\R^2$.
Combining this with \eqref{e:kinfor}, we conclude that $\theta$ is continuous by  \cite[Proposition~6]{DLOW03} (whose argument is recalled in the proof of Proposition~\ref{p:reg_curv_graph}), and hence so is $Dw=\gamma(\theta)$.
\end{proof}

\subsubsection{The case of a closed curve}

We focus now on the case $I=\R/L\Z$.
The first step is to observe, as in \cite{LLP25},
that the map $v=\Psi(\theta)\colon \Omega\to \mathbb{RP}^1$, 
which indicates the direction of characteristics, 
is a zero-state of the unoriented Aviles-Giga functional, as defined in \cite{GMPS24}.

\begin{lem}\label{l:ent_zero_v}
The map $v$ satisfies
\begin{align}\label{eq:ent_zero_v}
\dv\Phi(v)=0\quad\forall \Phi\in C^1(\mathbb S^1,\R^2)
\text{ even such that }\frac{d}{dt}\Phi(e^{it})\cdot ie^{it}=0
\quad\forall t\in\R\,.
\end{align}
\end{lem}
\begin{proof}
Denote $\Psi(t)=e^{i\psi(t)}$ with $\psi\in C^0(\R)$ such that $\psi(t+L)=\psi(t)+k\pi$ for some $k\in\Z$.
Define also
\begin{align*}
\lambda(e^{it})=e^{it}\cdot \frac{d}{dt}\Phi(e^{it})\,,
\end{align*}
so that $\lambda\in C^0(\mathbb S^1)$ is odd since $\Phi$ is even.
If $\Gamma$ is $C^2$, then $\Psi$ and $\psi$ are $C^1$, 
and we may simply apply Proposition~\ref{p:ent} to
$\Sigma=\Phi\circ\Psi\circ\gamma^{-1}$, 
which satisfies  
\begin{align*}
(\Sigma\circ\gamma)'(t)=(\Phi\circ\Psi)'(t)=\lambda(e^{i\psi(t)})\psi'(t)e^{i\psi(t)}\,,
\end{align*}
and belongs therefore to $\mathcal E_\Sigma$ by definition of $\Psi=e^{i\psi}$. Given that $\Gamma$ is merely $C^1$, we cannot use this direct calculation.
It would be natural to argue by approximation, but it is not clear how to construct a sequence of entropies $\Sigma_j\in\mathcal E_\Gamma$ converging pointwise to $\Sigma$.
We rely instead on the removability of $\Sing(w)$ as in Proposition~\ref{p:ent} and on the kinetic formulation obtained in Proposition~\ref{p:reg_open_curve} for nonclosed curves.\ Thus, as in Proposition~\ref{p:ent}, it suffices to show that, for fixed $x \in \Omega \setminus \Sing(w)$:
\begin{equation}\label{e:rem}
 \dv\Phi(v)=0, \quad \text{ in }\mathcal{D}'(B_r(x)) \text{ for some }r = r(x) > 0\,. 
\end{equation}
%
We choose $r>0$ small enough that 
$Dw(B_r(x))\subset\gamma([a,b])$ for some $a<b<a+L$
such that $\Psi$ admits a continuous lifting from $[a,b]$ into the arc $A=\lbrace e^{is}\colon s_1<s<s_2\rbrace$ for some $s_1<s_2<s_1+\pi$.
We write $\Psi=e^{i\psi}$ on $[a,b]$, with 
$\psi\in C^0([a,b])$ such that $s_1<\psi<s_2$.
Then, by the proof of Proposition~\ref{p:reg_open_curve},
the function $\theta\colon\Omega\to [a,b]$ such that $Dw=\gamma(\theta)$ satisfies \eqref{e:kinfor}, namely
\begin{align*}
\dv\big(
\Psi(\alpha)\mathbf 1_{\{\theta>\alpha\}}
\big) =0\quad\text{in }\mathcal D'(B_r(x))\,, \quad \text{ for all $\alpha\in [a,b]$.}
\end{align*}
Similarly, we get that $\Psi(\alpha)\mathbf{1}_{\{\theta \leq\alpha\}}$ and $\Psi(\alpha) \mathbf{1}_{\{\theta<\alpha\}}$ are divergence-free in $B_r(x)$. 
Moreover, we can assume, without loss of generality, that $|\{\theta = a\}\cup \{\theta = b\}| = 0$.
 These observations lead to
\begin{equation}\label{e:div0}
\dv\big(\xi \mathbf{1}_{\{\alpha<\theta<\beta\}}\big)=0
\quad\text{in }\mathcal D'(B_r(x))\,,\quad \text{for all }
\xi\in A
\text{ and }\alpha, \beta\in \lbrace a,b\rbrace \cup\Psi^{-1}(\lbrace\xi\rbrace),\; \alpha < \beta.
\end{equation}
Given $s\in (s_1,s_2)$ and $\xi=e^{is}$, the open set $\lbrace t: \psi(t)>s\rbrace\subset (a,b)$ is a countable union of intervals $(\alpha_j,\beta_j)$ with
$\alpha_j < \beta_j\in \lbrace a,b\rbrace \cup\Psi^{-1}(\lbrace\xi\rbrace)$.
By dominated convergence and \eqref{e:div0} we infer
\begin{equation}\label{e:diveis}
\dv\left( 
e^{is} \mathbf{1}_{\{\psi(\theta)>s\}} \right) =0 \quad\text{in }\mathcal D'(B_r(x))\,,\quad\forall s\in (s_1,s_2)\,.
\end{equation}
Moreover for $s\in (s_1,s_2)$ we have
\begin{align*}
\Phi(e^{is})=\Phi(e^{is_1})+\int_{s_1}^{s_2} \lambda(e^{i\tau})\mathbf{1}_{\{\tau<s\}}e^{i\tau}\,d\tau\,,
\end{align*}
so we deduce, for any $\zeta\in C_c^1(B_r(x_0))$ using Fubini's theorem,
\begin{align*}
\langle \dv\Phi(v),\zeta\rangle
&
=
\int_{s_1}^{s_2}\lambda(e^{i\tau})
\langle \dv\big(e^{i\tau}\mathbf{1}_{\{\psi(\theta)>\tau\}}\big),\zeta\rangle \, d\tau
\overset{\eqref{e:diveis}}{=}0\,,
\end{align*}
thus concluding the proof.
\end{proof}

Lemma~\ref{l:ent_zero_v} allows us to use \cite[Theorem~6.5]{GMPS24}, 
 which we recall for the reader's convenience.

\begin{thm}[\cite{GMPS24}]
\label{thm:GMPS}
If $v\colon\Omega\to\mathbb{RP}^1$ satisfies \eqref{eq:ent_zero_v},
then it has the following properties.
\begin{enumerate}
	\item The map $v$ is locally Lipschitz in $\Omega\setminus \mathcal{S}_v$ for a locally finite set $\mathcal{S}_v \subset \Omega$.
	\item For $x \in \Omega \setminus \mathcal{S}_v$, $v \equiv v(x)$ on the connected component of $[x + \mathbb{R}v(x)]\cap (\Omega\setminus \mathcal{S}_v)$ containing $x$;
	\item\label{enu:3} if $B \doteq B_r(x_0)$, $B_{2r}(x_0) \subset \Omega$, and $2B \cap \mathcal{S}_v = \{x_0\}$, then either
	\begin{enumerate}
		\item $v(x) = V^{x_0}(x)\doteq\frac{x-x_0}{|x-x_0|}$ in $B\setminus \{x_0\}$;
		\item or there exists $\xi \in \mathbb{S}^1$ such that 
		\begin{itemize}
			\item $v(x) = V^{x_0}(x)$ in $\{x \in B\colon (x-x_0,\xi) > 0\}$;
			\item $v$ is Lipschitz in $\{x \in B\colon (x-x_0,\xi) < 0\}$.
		\end{itemize}
	\end{enumerate}
\end{enumerate}
\end{thm}

\begin{rem}\label{r:v^2}
Here and in what follows we implicitly identify a map $v\colon\Omega\to\mathbb{RP}^1=\mathbb S^1/\lbrace\pm 1\rbrace$ with
any lifting $\tilde v\colon\Omega\to\mathbb S^1$ such that $v=\lbrace\pm \tilde v\rbrace$. 
In \cite[Theorem~6.5]{GMPS24}, this result is stated for the map $\tilde v^2$ which uniquely determines $v$ and still satisfies \eqref{eq:ent_zero_v}. 
Theorem~\ref{thm:GMPS} is a direct translation of that statement, taking into account the aforementioned implicit identification.
\end{rem}

\begin{prop}\label{p:reg_closed_curve}
	Assume $I=\R/L\Z$ and $w$ solves \eqref{e:GAMMA}, for $\Gamma$ fulfilling \eqref{e:diffinc}.\ Then, the field $v = \Psi(\theta)$ of characteristic lines is continuous in $\Omega\setminus\s_v$, where $\s_v$ is locally finite.\ Moreover, $\s_v =\Sing(w)$.
\end{prop}
\begin{proof}
Lemma~\ref{l:ent_zero_v} and Theorem \ref{thm:GMPS} imply that the field $v$ of characteristic lines is continuous in $\Omega\setminus\s_v$, where $\s_v$ is locally finite.
Let us show that $\s_v=\Sing(w)$.\ Clearly $\s_v\subset \Sing(w)$, since $v=\Psi\circ\gamma^{-1}(Dw)$ and $\Psi\circ\gamma^{-1}$ is continuous by Lemma \ref{l:psi}.
To show the other inclusion we fix $x_0\in\Omega\setminus \s_v$ and argue that $x_0$ must be a regular point of $w$.\ Since $Dw = \gamma(\theta)$ and $v = \Psi(\theta)$ we can write, for all $\delta > 0$, 
\[
[Dw](B_\delta(x_0))\subset \gamma(\Psi^{-1}(A_\delta)), \text{ where $A_\delta= v(\overline {B_\delta(x_0)})$}.
\]
By continuity of $v$ at $x_0$, $\bigcap_{\delta>0}A_\delta =\lbrace v(x_0)\rbrace$, and hence $[Dw](x_0)\subset\gamma(\Psi^{-1}(\lbrace v(x_0)\rbrace)).$
By Theorem \ref{theo:connected} and \eqref{e:diffinc} we infer that $[Dw](x_0)$ is connected, 
so it is contained in a single connected component of 
$\gamma(\Psi^{-1}(\lbrace v(x_0)\rbrace))$,
which must be of the form $\gamma(\mathcal C)$ for a single connected component $\mathcal C$ of $\Psi^{-1}(\lbrace v(x_0)\rbrace)\subset I$.
Thus $\mathcal C$ is an interval, and since $\Psi$ is not constant on any open interval by Lemma~\ref{l:psi}, we get that $\mathcal C$ is a singleton, thus $\diam([Dw](x_0))=0$.\ This yields $x_0\in\Reg(w)$ and concludes the proof.
\end{proof}

\subsection{Step 6: Structure of the singular set}\label{sub:step6}

Elementary geometric considerations give a quite strong constraint on the structure of the characteristic lines.
To see this, let $\theta\colon\Omega\to I$ 
be such that $Dw=\gamma(\theta)$,
and let $v=\Psi\circ\theta\colon\Omega\to\mathbb{RP}^1$, where $\Psi$ is the characteristic direction defined in Lemma~\ref{l:psi}.
The map $v$ is continuous in $\Omega\setminus\s_v$
and locally constant in its own direction.

\begin{lem}\label{l:struct_isolated_sing}
If $\Omega\subset\R^2$ is convex and 
$v\colon\Omega\to\mathbb{RP}^1$ is continuous and locally constant in its own direction in $\Omega\setminus\s_v$, for a locally finite set $\s_v$ of discontinuity points, then $\s_v$ contains at most two points, and for any $x_0\in\s_v$ there exists $\xi_0\in\mathbb S^1$ such that $v(x_0+x)=x/|x|$ for all $x\in \Omega-x_0$ with $x\cdot\xi_0>0$.
\end{lem}
\begin{proof}
From Theorem \ref{thm:GMPS}, we know that there exist $\delta>0$ and $\xi_0\in\mathbb S^1$ such that
\begin{equation}\label{eq:VOmega+}
v(x)=V^{x_0}(x) = \frac{x-x_0}{|x-x_0|}
\text{ for all }x\in B_\delta(x_0)\text{ with }(x-x_0)\cdot\xi_0>0\,.
\end{equation}
The following crucial property is used to show \eqref{eq:VOmega+} and we will make use of it in our proof as well: due to the convexity of $\Omega$, for any $x\neq y\in\Omega\setminus\s_v$, if the characteristic lines $x+\R v(x)$ and $y+\R v(y)$ intersect at a single point $z\in\Omega$, then at least one of the segments $[x,z]$ and $[y,z]$ contains a singular point.\ 

As a preliminary step in our proof, we claim that formula \eqref{eq:VOmega+} is actually valid in the whole half-domain 
\begin{align*}
\Omega_{x_0,\xi_0} \doteq \lbrace x\in\Omega\colon (x -x_0)\cdot\xi_0> 0\rbrace\,.
\end{align*}
By convexity of $\Omega$
and local constancy of $v$ in its own direction,
this is equivalent to $\Omega_{x_0,\xi_0}\cap\mathcal S_v=\emptyset$.
Assume, by contradiction, that there exists $x_1\in \Omega_{x_0,\xi_0}\cap\s_v$.
We can choose $x_1$ to be the closest element of $\s_v$ to $x_0$, and hence a neighborhood of $[x_0,x_1]$ contains only $x_0,x_1$ as singular points.
Employ again Theorem~\ref{thm:GMPS}  at $x_1$:
 there exist $\delta_1>0$ and $\xi_1\in\mathbb S^1$ such that $B_{\delta_1}(x_1)\cap\s_v=\lbrace x_1\rbrace$
and $v=V^{x_1}$ in $B_{\delta_1}(x_1)\cap\Omega_{x_1,\xi_1}$.
For small $\e>0$ and any $y\in B_{\e}(x_1)\cap\Omega_{x_1,\xi_1}$ such that 
$y-x_1$ is not parallel to $x_1-x_0$, the segment $[y,x_0]$ contains no singular point, so $v$ must be constant along it,
that is, equal to $(y-x_0)/|y-x_0|$, 
in contradiction with the form of $v$ near $x_1$.\ Hence our claim holds: \eqref{eq:VOmega+} is valid in $\Omega_{x_0,\xi_0}$.\ 

Assume now that $\s_v$ contains at least two singular points $x_1\neq x_2\in\s_v$.\ Then there exist $\xi_1,\xi_2\in\mathbb S^1$ such that $v=V^{x_j}$ in $\Omega_{x_j,\xi_j}$ for $j=1,2$.
This implies in particular that the two sets $\Omega_{x_1,\xi_1}$ and $\Omega_{x_2,\xi_2}$ must be disjoint.\ Consider, for $j=1,2$, the two lines $L_j=x_j+\R\xi_j^\perp$ which bound $\Omega_{x_j,\xi_j}$.\ Assume by contradiction that there exists a third, distinct singular point $x_3$, and let $\xi_3\in\mathbb S^1$ be such that $v=V^{x_3}$ in $\Omega_{x_3,\xi_3}$.\ Clearly, this set must be disjoint from $\Omega_{x_j,\xi_j}$, $j=1,2$.\ From this, we deduce immediately that $x_3 \notin L_j$.\ We notice that if $L_1 = L_2$, then $\Omega = \overline{\Omega_{x_1,\xi_1}}\cup\overline{\Omega_{x_2,\xi_2}}$.\ In that case, since $(L_1\cup L_2)\cap \mathcal{S}_v = \{x_1,x_2\}$, $\s_v=\lbrace x_1,x_2\rbrace$.\ We can therefore assume that $L_1 \neq L_2$, so that the open set lying between them is nonempty:
\begin{align*}
U=\Omega\setminus \Big(
\overline{\Omega_{x_1,\xi_1}} \cup \overline{\Omega_{x_2,\xi_2}}) \neq \emptyset.
\Big)\,,
\end{align*}
Necessarily $x_3\in U \cap \s_v$, and we can assume without loss of generality that $x_3$ is the point in $U \cap \s_v$ closest to $[x_1,x_2]$.\ As $\Omega_{x_3,\xi_3}$ must be disjoint from $\Omega_{x_1,\xi_1}\cup \Omega_{x_2,\xi_2}$, $x_3$ does not belong to $[x_1,x_2]$, hence the open interval  $X\doteq [x_1,x_2]\setminus\lbrace x_1,x_2\rbrace$ does not contain any singular point (since $x_3$ is a singular point in $U$ closest to that segment), and $x_1,x_2$ are the only singular points in a neighborhood of the segment $[x_1,x_2]$.\ Assume without loss of generality that $[x_1,x_2]$ is horizontal, that $x_1$ lies left of $x_2$, and that $x_3$ lies above $[x_1,x_2]$.
Hence the lines $L_1,L_2$ enclosing $U$ are not horizontal.
For $x\in  X$, 
the line $\ell_x =x+\R v(x)$ cannot be horizontal: 
otherwise, the characteristic line 
from any point above and close enough to $x$ is nearly horizontal,
its intersection with $U$ is contained in the regular neighborhood of $[x_1,x_2]$, 
and intersects the characteristic lines $L_1$ and $L_2$ at regular points, which is impossible.
Hence, each line $\ell_x$, for 
 $x\in  X$,
 intersects the horizontal line $H$ passing through $x_3$ at a point $y(x)$.\ Since $v$ is continuous along $X$, the map $x\mapsto y(x)$ is continuous on $X$.\ As $\lim_{x \to x_{i}, x \in X}v(x) = \xi_i^\perp$, $y(x)$ lies left/right of $x_3$ on $H$ for $x \in X$ sufficiently close to $x_1$/$x_2$.\ As a consequence, there exists $x\in X$ such that $y(x)=x_3$.
This yields a characteristic segment $\hat L$ connecting $x$ to $x_3$.\ Set $L_3= x_3+\R\xi_3^\perp$. 
Observe that $x \in \Omega_{x_3,-\xi_3}$ and that $\hat L \cap L_3 = \{x_3\}$, otherwise $x \in \overline{\Omega_{x_3,\xi_3}}$ and thus either $x_1$ or $x_2$ belongs to $\Omega_{x_3,\xi_3}$, a contradiction.\ Therefore, $\hat L \subset \overline{\Omega_{x_3,-\xi_3}}$, and $\hat L \cap L_3 = \{x_3\}$.\ This implies that all characteristic lines starting from 
$z\in \Omega_{x_3,-\xi_3}$ close enough to $x_3$ 
must stay in a sector delimited by $\hat L$ and $L_3= x_3+\R\xi_3^\perp$, and must therefore intersect $x_3$.\ Hence, $v=V^{x_3}$ in $B_r(x_3)\cap \Omega_{x_3,-\xi_3}$ for some small $r > 0$ and the same holds for $\Omega_{x_3,\xi_3}$ by the first claim of the proof.\ For the same reason, $v = V^{x_3}$ in $\Omega$, in contradiction with the fact that $\lbrace x_1,x_2\rbrace\subset\s_v$. Thus such $x_3$ could not exist, and we conclude the proof.
\end{proof}

\begin{cor}\label{c:struct_isolated_sing_R2}
If $v$ is as in Lemma~\ref{l:struct_isolated_sing} and $\Omega=\R^2$, 
then there exist $s=(s_1,s_2),t = (t_1,t_2)$, $s_1\in [-\infty,+\infty)$, $t_1\in (-\infty,+\infty]$ with $s_1\leq t_1$, $t_2,s_2\in \R$, and $Q\in SO(2)$ such that for $x \neq s,t$
\begin{equation}\label{e:charline}
Q^Tv(Qx)
=\begin{cases}
\frac{x-(s_1,s_2)}{|x-(s_1,s_2)|}
&
\quad\text{ if }x_1<s_1\,,
\\
 e_2 
&
\quad\text{if }s_1\le x_1\le t_1\,,\\
\frac{x-(t_1,t_2)}{|x-(t_1,t_2)|}
&
\quad\text{if }x>t_1\,.
\end{cases}
\end{equation}
\end{cor}

\begin{rem}\label{r:struct_sing_R2}
Corollary~\ref{c:struct_isolated_sing_R2} characterizes all possible entire configurations of characteristic lines with locally finite singular set, as  members of a 5-dimensional family.
Special cases are the
constants, corresponding to $t_1=-s_1=+\infty$,
the
 single half-vortices, corresponding to 
 $(s_1,t_1)\in (\lbrace -\infty\rbrace\times \R ) \cup
 (\R\times\lbrace +\infty\rbrace)$, and the single vortices, corresponding to $(s_1-t_1,s_2-t_2)=(0,0)$.
\end{rem}

\begin{proof}[Proof of Corollary~\ref{c:struct_isolated_sing_R2}]
We distinguish three cases depending on the cardinality of $\s_v$.\

If $\s_v=\emptyset$ then the characteristic lines cannot intersect and must then all be parallel to a single direction $Q e_2$ for some $Q\in SO(2)$, which corresponds to the case $t_1=-s_1=+\infty$.

If $\s_v=\lbrace x_0\rbrace$, then there exists $\xi_0\in\mathbb S^1$ such that 
$v=V^{x_0}$ in $\Omega_{x_0,\xi_0}$.\ Writing $\xi_0=-Q e_1$ for some $Q\in SO(2)$ and replacing $v$ by $Q^T v\circ Q$ we assume without loss of generality that $\xi_0=- e_1$.
Hence $v(x)=V^{x_0}(x)$ for $x_1<s_1\doteq x_0\cdot e_1$. 
Let $s_2=x_0\cdot e_2$, so that $x_0=(s_1,s_2)$.
The characteristic lines starting from any $x\in\R^2$ with $x_1>s_1$ 
can only intersect the characteristic vertical line $\lbrace x_1=s_1\rbrace$ at $x=x_0$.
So they must either all be vertical, namely $t_1 = + \infty$, or all pass through $x_0$, i.e. $(t_1,t_2)=(s_1,s_2)$.

Assume finally that $\s_v=\lbrace x_1,x_2\rbrace$.
Let $\xi_j\in\mathbb S^1$ be such that $v=V^{x_j}$ in $\Omega_{x_j,\xi_j}$ for $j=1,2$.
Since these two half-planes must be disjoint, their boundaries are parallel, hence $\xi_2=-\xi_1$. 
Applying a rotation, we assume without loss of generality that $\xi_2=-\xi_1 = e_1$.
Then $v(x)$ is as in \eqref{e:charline} for $x_1<s_1$ and $x_1> t_1$, where $s_k=x_1\cdot e_k$ and $t_k=x_2\cdot e_k$ for $k=1,2$.
If $s_1=t_1$ we are done.\ If $s_1< t_1$,  the characteristic lines starting from any point $x \in \lbrace s_1< x_1 < t_1\rbrace$
can intersect the vertical characteristic lines which form the stripe's boundary only at $x_1$ or $x_2$. But if such an intersection happens, then $v$ must be a vortex in $\R^2$, and $\s_v$ contains only one element. So they must all be vertical, and we conclude the proof.
\end{proof}

\section{Solutions of degenerate equations}\label{s:degen_eq}

In this final section, we show a stronger and more precise version of Theorem \ref{t:Dintro}, Theorem \ref{tint:4}.\ First of all, let us introduce the degeneracy set $\mathcal{D}$ in full generality.\ 
Following \cite{dSS10,Lacombe2024},
 for a strictly monotone $G\in C^0(\R^2,\R^2)$ we define its degeneracy set $\mathcal D \doteq \mathcal D_-\cap\mathcal D_+$ which can be interpreted as the set of points 
where the symmetric parts of $DG$ and $(DG)^{-1}$ both have zero eigenvalues.\ More precisely, the sets $\mathcal D_\pm$ are:
\begin{align}\label{e:Dpm}
	\begin{aligned}
		\mathcal D_-
		&
		=
		\mathcal D_-(G)
		=\bigcap_{\lambda>0}
		\overline{ \left\lbrace                                                                             
			x\in \R^2\colon \liminf_{h\to 0} \frac{( G(x+h)-G(x),h)}{|h|^2} \leq\lambda\right\rbrace}\,,
		\\
		\mathcal D_+
		&
		=
		\mathcal D_+(G)=
		\bigcap_{\lambda>0}
		\overline{\left\lbrace
			x\in \R^2\colon \liminf_{h\to 0} \frac{( G(x+h)-G(x),h )}{|G(x+h)-G(x)|^2} \leq\lambda\right\rbrace}\,.
	\end{aligned}
\end{align}
They correspond to the smallest closed sets outside which $G$ is locally elliptic from below or above:
\begin{align*}
	x\in\mathcal D_-^c\quad\Longleftrightarrow
	\quad
	&
	\exists\lambda,\delta>0\,:\; (G(x_2)-G(x_1),x_2-x_1)\geq \lambda |x_2-x_1|^2
	\quad\forall x_1,x_2\in B_\delta(x)\,,
	\\
	x\in\mathcal D_+^c\quad\Longleftrightarrow
	\quad
	&
	\exists\lambda,\delta>0\,:\; (G(x_2)-G(x_1),x_2-x_1)\geq \lambda |G(x_2)-G(x_1)|^2
	\quad\forall x_1,x_2\in B_\delta(x)\,.
\end{align*}
The latter is also equivalent to $G(x)\in \mathcal D_-^c(G^{-1})$, 
justifying its interpretation as local  ellipticity from above.
We are ready to state our main result of this section:

\begin{thm}\label{tint:4} Let $G \in C^0(\R^2,\R^2)$ fulfill \eqref{e:monoG}, and assume that each connected component of $\mathcal D=\mathcal D_-\cap\mathcal D_+$ has image through the graph map 
	\begin{align*}
		\R^2\ni x\mapsto \left(\begin{array}{c}x\\ G(x)\end{array}\right)\in \R^{2\times 2}\,,
	\end{align*}
	contained in a $C^1$ curve. 
	If all but a finite number of these components are
	\begin{align}\label{e:cond_reg_curve}
		\text{either simply connected 
			or boundaries of strictly convex open sets,}
	\end{align}
	then any Lipschitz solution $u$ of \eqref{e:G} is $C^1$ outside a locally finite singular set.
	Moreover, that singular set is empty if all components satisfy
	\eqref{e:cond_reg_curve}.
\end{thm}
Crucial to our proof is the following result of \cite{lacombe26}, which shows that blowup limits of solutions to \eqref{e:G} are either affine or take values into $\mathcal D$.\ Recalling from Proposition~\ref{p:11} the correspondence between \eqref{e:G} and \eqref{e:di}, there is a direct reformulation in terms of differential inclusions, which will be more convenient for us.

\begin{thm}[\cite{lacombe26}]\label{t:blow_ups_K*}
Let $K\subset\R^{2\times 2}$ satisfy \eqref{e:diffinc},
 $w\in \Lip( B_1,\R^2)$ solve $Dw\in K$ a.e. in $B_1$,
 and $x_0\in B_1$.\ Any blowup limit $w_\infty\in \mathcal B(w)(x_0)$
is either affine or satisfies $Dw_\infty\in \mathcal K_*$ a.e., where
\begin{align}
\label{e:K*}
\mathcal K_*
&
=\mathcal K_{*}^{1}\cap\mathcal K_{*}^{2}
\,,
\qquad
\mathcal K_*^j
=
\bigcap_{\lambda>0}
\overline{ \left\lbrace
A\in K\colon 
\liminf_{K\ni B\to A} 
\frac{\det(A-B)}{|L_j(A-B)|^2} \leq\lambda\right\rbrace}\,,
\end{align}
and $L_j(M)$ denotes the $j$-th row of a matrix $M$, for $j=1,2$.
\end{thm}

In this section, we combine this property of blowup limits with Theorems \ref{t:summa}-\ref{tint:3}, Corollary \ref{cor:2} and structural assumptions on the degenerate set $\mathcal D$, or equivalently $\mathcal K_*$, in order to deduce partial regularity properties of \eqref{e:di}, or equivalently \eqref{e:G}, obtaining in particular Theorem~\ref{tint:4}.

\subsection{Regularity threshold}

For any $K\subset\R^{2\times 2}$, let us introduce the nonnegative number
\begin{align}\label{e:eps*}
\e_*(K)
\doteq
\sup
\big\lbrace \e\geq 0\colon \big(
Dw\in K\text{ a.e. in }B_1
\text{ and }\diam([Dw])\leq \e \big)
\Rightarrow w\in C^1(B_1)\big\rbrace\,,
\end{align}
which is the threshold for regularity of solutions of \eqref{e:di} with small gradient oscillations.
With this notation, the first statement in  Proposition~\ref{p:reg_curve_small_osc_ptwise} 
can be reformulated as
 $\e_*(\Gamma)>0$ for any compact connected $C^1$ curve $\Gamma$ satisfying \eqref{e:diffinc}.
Moreover, Corollary~\ref{c:visc} and Proposition~\ref{p:reg_open_curve} imply $\e_*(\Gamma)=+\infty$ if there exists $a\in\mathbb S^1$ such that the projection $a\Gamma\subset\R^2$ is the boundary of a strictly convex open set or if $\Gamma$ is simply connected.\ We can also reformulate partial regularity in terms of positivity of $\e_*(K)$:
 

\begin{lem}\label{l:partial_reg_eps*}
Let $K\subset\R^{2\times 2}$ satisfy \eqref{e:diffinc} and $\e_*(K)>0$. For any Lipschitz solution $w$ of \eqref{e:di}, 
we have that $\Sing(w)$ is closed,
$\mathcal H^1(\Sing(w))=0$ and $w\in C^1(\Omega\setminus \Sing(w))$.
\end{lem}
\begin{proof}
Thanks to 
Theorem \ref{t:summa}, it suffices to show that the set $\Reg(w)$ defined in Definition~\ref{def:reg} is open. Let $\e = \e_*(K)/2>0$. 
By 
Theorem \ref{t:noninc}, if $x\in \Reg(w)$ then there exists $\delta>0$ such that $\diam([Dw](B_\delta(x)))\leq \e$, so $w$ is $C^1$ in $B_\delta(x)$ by definition \eqref{e:eps*} of $\e_*(K)$, hence $B_\delta(x)\subset\Reg(w)$. 
\end{proof}

Thanks to Theorem~\ref{t:blow_ups_K*} and 
Theorem~\ref{t:noninc}, 
we can relate the regularity threshold $\e_*$ of $K$ with the regularity thresholds of the connected components of $\mathcal K_*$.

\begin{prop}\label{p:eps*}
For any $K\subset\R^{2\times 2}$ satisfying \eqref{e:diffinc}, let $\mathcal K_*$ as in \eqref{e:K*}.\ Then,
\begin{align*}
\e_*(K)\geq
\inf
\big\lbrace \e_*(\mathcal C) \colon \mathcal C 
\text{ connected component of }\mathcal K_*
\big\rbrace\,.
\end{align*}
\end{prop}
\begin{proof}
If the infimum is zero, there is nothing to show. Assume therefore that it is positive, and fix
$0<\e<\inf\lbrace\e_*(\mathcal C)\rbrace$.
Let $w\colon B_1\to \R^2$ satisfy $Dw\in K$ a.e. and $\diam([Dw](B_1))\leq \e$.
This pointwise constraint is preserved under blowup: for any $x_0\in B_1$ and $w_\infty\in \mathcal B(w)(x_0)$ we have $\diam([Dw_\infty])(B_1)\leq \e$.
By Theorem~\ref{t:blow_ups_K*}, the blowup limit $w_\infty$ is either affine or satisfies $Dw_\infty\in \mathcal K_*$ a.e., hence $Dw_\infty\in \mathcal C$ a.e. for some connected component $\mathcal C$ of $\mathcal K_*$ by Theorem \ref{theo:connected}. 
By definition of $\e_*(\mathcal C)$ this implies $w_\infty\in C^1(B_1,\R^2)$, hence $x_0\in\Reg(w)$ by Definition \ref{def:reg}.
Thus $\Reg(w)=B_1$ and $\e_*(K)\geq \e$.
\end{proof}

Of course, Proposition~\ref{p:eps*} is only interesting if we know that the connected components of $\mathcal K_*$ have a positive regularity threshold \eqref{e:eps*}.
Thanks to  Proposition~\ref{p:reg_curve_small_osc_ptwise}, this is the case if we assume that
\begin{align}\label{eq:hyp_K*1}
\begin{aligned}
&\text{every connected component of }\mathcal K_*
\text{ is included in a }C^1\text{ curve.}
\end{aligned}
\end{align}
Under this assumption, every connected component $\mathcal C$ of $\mathcal K_*$
is either a point, or satisfies $\mathcal C\subset \Gamma = \gamma(I)$ for some $I=[a,b]$ or $\R/L\Z$ and $\gamma\colon I\to\Gamma$ a $C^1$ homeomorphism with $|\gamma'|>0$. In the latter case,
since $\mathcal C$ is compact and connected and $\gamma$ is a homeomorphism, we must in fact have $\mathcal C = \gamma(J)$ for some
compact interval $J\subset I$.
So 
 \eqref{eq:hyp_K*1} is equivalent to 
every connected component $\mathcal C$ of $\mathcal K_*$ being either a point
 or a compact $C^1$ curve,
which implies $\e_*(\mathcal C)>0$ by Proposition~\ref{p:reg_curve_small_osc_ptwise}.
In fact, we will see in \textsection\ref{sec:charac} that under assumption \eqref{eq:hyp_K*1}, the property $\e_*(K)>0$ implies a much stronger conclusion than that of Lemma~\ref{l:partial_reg_eps*}.

\begin{rem}\label{r:eps*} Condition \eqref{e:diffinc} is not sufficient to ensure $\e_*(K)>0$.\ Consider indeed, for any $n \in \N$,
\begin{equation}\label{e:hw}
f_n(z) \doteq \frac{z^n}{|z|^{n - 1}}, \; h(z) \doteq -\frac{1}{3}f_3(z), \; w(z) \doteq f_2(z).
\end{equation}
Direct computations yield
\begin{equation}\label{e:dc}
\partial_z f_n(z) = \frac{n + 1}{2}\frac{z^{n-1}}{|z|^{n-1}} \text{ and } \partial_{\bar z} f_n(z) = \frac{1-n}{2}\frac{z^{n+1}}{|z|^{n +1 }} 
\end{equation}
so that $w$ solves \eqref{eq:g} for $h$ defined in \eqref{e:hw}.
Through \eqref{e:dc}-\eqref{alg}, we deduce that $w$ is Lipschitz.\ Moreover, for any $v \in \mathbb{S}^1$, \eqref{e:dc} and \eqref{e:fund} imply that
\[
|Dh(z)v|^2= \left|\frac{2}{3}v - \frac{1}{3}\frac{z^2}{|z|^2}\bar v \right|^2 =\frac{5}{9} -\frac{4}{9}\re\left(\frac{z^2}{|z|^2}(\bar v)^2\right) \le 1, \text{ with equality if and only if $(v,z) = 0$}.
\]
In particular, if $a, b \in \mathbb{C}$ are such that $|h(b)-h(a)| = |b-a|$, we see that
\[
(b-a,a + t(b-a)) = 0, \text{ for all }t \in [0,1]\,.
\]
This can only happen if $a = b$, and hence $h$ fulfills \eqref{e:contrac}.\ Since $h$ is one-homogeneous, $w_\alpha(z) \doteq \alpha w(z)$ solves \eqref{eq:g} for $h$ as in \eqref{e:hw} for all $\alpha \in \R$.\ Moreover $\Sing(w_\alpha) = \{0\}$, regardless of how small $\alpha$ is, showing that no $\eps$-regularity theorem can hold for the PDE \eqref{eq:g} for such $h$.
In other words, the regularity threshold \eqref{e:eps*} of the  set $K=\lbrace (a,h(a))\colon |a|\leq 1\rbrace$ is $\e_*(K)=0$.
 The problem in this case is, in fact, fully degenerate: for $h$ as in \eqref{e:hw}, passing to the equivalent formulation in terms of a monotone field $G$ as in \eqref{e:G} (see Propositions \ref{p:11}-\ref{p:12}), we find that $\mathcal{D} = \mathcal{D}_+\cap \mathcal{D}_- = \mathbb{\R}^2$, or equivalently $K=\mathcal K_*$.\ 
 To see this, one may use the calculations in the proof of \cite[Theorem 1.5]{Lacombe2024}, which yield $F(\mathbb{S}^1) \subset \mathcal{D}$, where $F$ is the homeomorphism
 \[
 F: \mathbb{C} \to \mathbb{C}, \quad F(z) \doteq \frac{h(z) + \bar z}{2}.
 \]
 Since $h$ is 1-homogeneous, the same calculations show that $F(t\mathbb{S}^1) \subset \mathcal{D}$ for all $t > 0$, and, since $F(\mathbb{C}) = \mathbb{C}$ and $\mathcal{D}$ is closed, we get that $\mathcal{D}$ must be, in this case, the whole plane.
\end{rem}

\subsection{Characterization of blowup limits}\label{sec:charac}

Since blowup limits $w_\infty$ are entire solutions of $Dw_\infty\in \mathcal C$ a.e. for some connected component $\mathcal C$ of $\mathcal K_*$, 
under assumption \eqref{eq:hyp_K*1} 
we know by Theorem~\ref{tint:3} that they are $C^1$ away from at most two singular points.
This statement is valid for any entire solution, 
but blowup limits have the additional property that their Hessian determinant vanishes away from the origin
 by 
 Corollary~\ref{cor:2}.
 Using this, we establish that they  have in fact at most one singularity.

\begin{prop}\label{p:blowups_hypK*}
Assume \eqref{e:diffinc} and \eqref{eq:hyp_K*1}.
Let $w\in \Lip(B_1,\R^2)$ solve $Dw\in K$ a.e., and $x_0\in B_1$.
Any blowup limit
$w_\infty\in \mathcal B(w)(x_0)$ can only be singular at $0$ and $Dw_\infty$ is 0-homogenenous.
\end{prop}


We start by ruling out
specific singularities  if the weak Hessian determinant vanishes.

\begin{lem}\label{l:hom_hessian}
Let $u\in C^1(B_1\setminus\lbrace 0\rbrace)$ be Lipschitz,  $1$-homogeneous,
 and satisfy $\mathcal D(u,u)=0$ in $B_1$.
If $u$ coincides with a linear function on $B_1^-=B_1\cap\lbrace x_1 < 0\rbrace$, then $u$ is linear.
\end{lem}
\begin{proof}
Write $u(re^{i\theta})=rf(\theta)$ for some $f\in C^1(\T)$, $\T=\R/2\pi\Z$.
Recalling \eqref{e:vwj}, $\mathcal D(u,u)=0$ is equivalent to
\begin{equation}\label{eq:D22}
\int_{B_1} ( D^2\varphi JDu,JDu)\, dx=0,
\quad
\forall\varphi\in C_c^2(B_1)\,.
\end{equation}
Since $JDu$ is divergence-free, this identity is preserved after subtracting a linear function from $u$, so we may without loss of generality assume that $u=0$ on $B_1^-$, that is, $f=0$ on $[\pi/2,3\pi/2]$.
Choosing radial test functions $\varphi(re^{i\theta})=\psi(r)$ 
with $\psi\in C_c^2([0,1))$ such that $\psi'(0)=\psi''(0)=0$ and $\psi(0)=1$,
\eqref{eq:D22} becomes
\begin{align*}
0
&
=\int_0^1\int_{\T} 
\Big(
f'(\theta)^2\psi''(r) +f(\theta)^2\frac{\psi'(r)}{r}
\Big)
\,d\theta\, rdr
\\
&
=\int_{\T}f'(\theta)^2\,d\theta \int_0^1\psi''(r)\, rdr
+\int_{\T}f(\theta)^2\, d\theta \int_0^1 \psi'(r)\, dr
=\int_{\T}f'(\theta)^2 \,d\theta  -\int_{\T}f(\theta)^2\, d\theta\,.
\end{align*}
As $f\equiv 0$ on $[\pi/2, 3\pi/2]$, we see that $f$ fulfills the equality case in Wirtinger's inequality on $[-\pi/2,\pi/2]$ with Dirichlet conditions.\ We infer $f(\theta)=a\mathbf 1_{|\theta|<\pi/2}\cos\theta$ for some $a\in\R$, thus 
$f=0$ since $f$ is $C^1$ on $\T$.
%
%
%
\end{proof}

\begin{proof}[Proof of Proposition~\ref{p:blowups_hypK*}]
As argued above, by Theorem~\ref{t:blow_ups_K*} and Theorem~\ref{theo:connected}, 
under assumptions \eqref{e:diffinc}-\eqref{eq:hyp_K*1}, the blowup limit $w_\infty$ 
is either linear or it satisfies $Dw_\infty\in \Gamma$ a.e. in $\R^2$,
 where
 $\Gamma=\gamma(I)$ is a compact connected $C^1$ curve satisfying \eqref{e:diffinc}, as in \S~\ref{sec:curves}.\ 
We can assume we are in the latter case, as otherwise there is nothing to show, and for the same reason we can  assume that $0 \in \Sing(w_\infty)$, since otherwise $x_0\in\Reg(w)$ and $w_\infty$ is affine by 
Theorem \ref{t:summa}. 
Letting $\theta=\gamma^{-1}(Dw_\infty)$ and $\Psi$ as in Lemma~\ref{l:psi},
we have therefore that 
the field of characteristic lines $v = \Psi\circ \theta$ must be of the form \eqref{e:charline}.\ On any subset  $A\subset \R^2\setminus \Sing(w_\infty)$ on which $v$ is constant, 
say $v = v_0$, $Dw_\infty$ must belong to $\Psi^{-1}(\{v_0\})$, a totally disconnected set by Lemma \ref{l:psi}, and therefore be constant.
Combining this with the structure \eqref{e:charline} of the map $v$, we deduce that $Dw_\infty$ is either constant, or $0$-homogeneous, or that it has two discontinuity points, 0 and $y_0 \neq 0$.

To prove Proposition~\ref{p:blowups_hypK*} we just need to rule out the latter case.
To do so, by 
Theorem~\ref{t:summa} it suffices to show that there exists a linear blowup of $w_\infty$ at $y_0$.\ Let $W$ be any such blowup.\ Due to the structure of the characteristic lines of $w_\infty$ \eqref{e:charline}, we infer that $DW$ is $0$-homogeneous and constant in $P_\alpha \doteq \lbrace (x,e^{i\alpha})< 0\rbrace$ for some $e^{i\alpha}\in\mathbb S^1$.\ Up to a rotation, we can assume without loss of generality that $\alpha = 0$.\
Hence its first component, $u \in C^1(\R^2\setminus \lbrace0\rbrace)$, is $1$-homogeneous and coincides with a linear function on $P_0$.\ In addition, since $w_\infty \in \mathcal{B}(w)(x_0)$, then $y_0 \notin \spt(\mathcal{D}(w_\infty))$ by Corollary \ref{cor:2}, 
and hence for the same reason $\mathcal{D}(u,u) = 0$ in $\R^2$.\ Combining this information with the structure of $u$ we can infer from Lemma \ref{l:hom_hessian} that $u$ is linear, and hence $y_0 \notin \Sing(w_\infty)$, a contradiction.
%
\end{proof}

If we know in addition that the regularity threshold \eqref{e:eps*} of $K$ is positive, 
Proposition~\ref{p:blowups_hypK*} implies discreteness of the singular set.

\begin{cor}\label{c:sing_hypK*1}
Assume \eqref{e:diffinc}, \eqref{eq:hyp_K*1}, and $\e_*(K)>0$.
Then any  Lipschitz solution $w$ of \eqref{e:di} has a locally finite  singular set $\Sing(w)$.
\end{cor}
\begin{proof}
Let $x_0\in \Sing(w)$ such that $B_{\delta}(x_0)\subset\Omega$, and assume that $x_0$ is not isolated.\
Then there exists a sequence $x_j\in\Sing(w)\setminus\lbrace x_0\rbrace$ such that $x_j\to x_0$.\
Let $r_j=|x_j-x_0|\to 0$ and consider the maps $w_j \doteq w_{r_j,x_0}$, see \eqref{e:resc}.\
They satisfy $0,y_j\in\Sing(w_j)$ where $y_j=(x_j-x_0)/r_j\in\mathbb S^1$.
Consider a (non relabeled) subsequence such that
$w_j\to w_\infty\in \mathcal B(w)(x_0)$ and $y_j\to y_\infty\in \mathbb S^1$.
By 
Theorem~\ref{t:noninc} and the assumption $\e_*(K)>0$ we have $0,y_\infty\in \Sing(w_\infty)$, in contradiction with Proposition~\ref{p:blowups_hypK*}.
\end{proof}

Combining Corollary~\ref{c:sing_hypK*1} with Propositions \ref{p:eps*} and \ref{p:reg_curve_small_osc_ptwise}, we deduce the following regularity result:

\begin{cor}\label{c:sing_hypK*2}
Assume \eqref{e:diffinc}, \eqref{eq:hyp_K*1}, and 
that all but a finite number of connected components $\mathcal C$ of the degenerate set $\mathcal K_*$ satisfy $\e_*(\mathcal C)=+\infty$.
Then any $w \in \Lip(B_1,\R^2)$ solving \eqref{e:di} has a locally finite  singular set $\Sing(w)$.
\end{cor}

Finally, recalling that $\Gamma\subset\R^{2\times 2}$ has $\e_*(\Gamma)=+\infty$ if its first-row projection is the boundary of a strictly convex open set,
or if it is a compact and simply connected $C^1$ curve $\Gamma\subset\R^{2\times 2}$ satisfying \eqref{e:diffinc}, we see that Theorem~\ref{tint:4} follows from Corollary~\ref{c:sing_hypK*2} and Proposition~\ref{p:11}.

\bibliographystyle{acm}
\bibliography{diffinc}

\end{document}